\documentclass[a4paper, 11pt] {article}
\usepackage{amsfonts, amsmath, amssymb, amsthm,url}
\usepackage{graphicx}
\usepackage[english] {babel}
\usepackage{enumitem}

\usepackage[ps, all,cmtip, dvips]{xy}

\usepackage{geometry}
\newtheorem{thm}{Theorem}[section]
\newtheorem{cor}[thm]{Corollary}
\newtheorem{lem}[thm]{Lemma}
\newtheorem{defi}[thm]{Definition}
\newtheorem{prop}[thm]{Proposition}

\theoremstyle{definition}
\newtheorem{rem}[thm]{Remark}

\newcommand{\R}{\mathbb{R}}

\DeclareMathOperator{\Ker}{Ker}
\DeclareMathOperator{\Ima}{Im}

\newcommand{\Z}{\mathbb{Z}}

\newcommand{\Q}{\mathbb{Q}}

\DeclareMathOperator{\Th}{Th}
\DeclareMathOperator{\ind}{Ind}
\newcommand{\C}{\mathbb{C}}

\DeclareMathOperator{\Ext}{Ext}
\DeclareMathOperator{\Hom}{Hom}

\DeclareMathOperator{\indre}{int}

\DeclareMathOperator{\Sq}{Sq}
\newcommand{\A}{\mathcal{A}}

\title{Cobordism obstructions to independent vector fields}
\author{Marcel B\"{o}kstedt, Johan Dupont, and Anne Marie Svane}
\date{}
\begin{document}

\maketitle

\begin{abstract}{We define an invariant for the existence of $r$ pointwise linearly independent sections in the tangent bundle of a closed manifold. For low values of $r$, explicit computations of the homotopy groups of certain Thom spectra combined with classical obstruction theory identifies this invariant as the top obstruction to the existence of the desired sections. In particular, this shows that the top obstruction is an invariant of the underlying manifold in these cases, which is not true in general. The invariant is related to cobordism theory and this gives rise to an identification of the invariant in terms of well-known invariants. As a corollary to the computations, we can also compute low-dimensional homotopy groups of the Thom spectra considered in \cite{GMTW}.}
\end{abstract}


\section{Introduction}\label{obstruction}
Given a smooth compact connected oriented $d$-dimensional manifold $M$, it is a classical problem to determine the maximal number of pointwise linearly independent tangent vector fields on $M$. The solution is only known for a small number of vector fields or certain nice classes of manifolds, see e.g.\ \cite{thomas} for an overview. Some of the best known results were given by Atiyah and Dupont in \cite{AD} and \cite{dupont}, in which they find the complete conditions for the existence of up to three independent vector fields. In this paper, we are going to try a refinement of their approach.

We will assume that $M$ allows $r$ vector fields $s= \{s_1,\dots,s_r\}$ with only finitely many singularities. This is in general a non-trivial assumption. One way of ensuring this is by assuming that $M$ is $(r-2)$-connected and
\begin{align*}
 w_{d-r+1}(TM)=0 & \textrm{ if $d-r$ odd,}\\
\delta^*w_{d-r}(TM)=0& \textrm{ if $d-r$ even.}
\end{align*}
Here $w_i(TM)$ is the $i$th Stiefel--Whitney class of the tangent bundle and  $\delta^*$ is the Bockstein  map $\delta^*: H^{d-r}(M;\Z/2) \to H^{d-r+1}(M;\Z)$. 
The obstruction to the existence  of $r$ vector fields with only finitely many singularities is described in general in \cite{steenrod}.  

Choose small disjoint disks $D_i$ around each singularity. By a Gram--Schmidt argument, we may assume that the vector fields are orthonormal on the boundary $\partial D_i$. Since the tangent bundle is trivial over each $D_i$, the vector fields define maps
\begin{equation*}
s_i : \partial D_i \to V_{d,r}
\end{equation*}
where $V_{d,r}$ is the Stiefel manifold consisting of $r$-tuples of orthonormal vectors in $\R^d$. 
\begin{defi}
The index of $s$ is defined to be
\begin{equation*}
\ind(s) = \sum_i [s_i] \in \pi_{d-1}(V_{d,r}).
\end{equation*}
\end{defi}
This is the top obstruction to the existence of $r$ independent vector fields in the sense of \cite{steenrod}. That is, if $\ind(s)=0$, then $r$ independent vector fields do exist. However, $\ind(s)$ is not an invariant. It depends on the chosen $s$. Hence it may be non-zero even though $M$ does allow $r$ independent vector fields. The goal of this paper is to find conditions under which the index is an invariant of $M$. In this case, the desired vector fields exist if and only if $\ind(s)$ vanishes. 

 For $r=1$ it was shown by Hopf in \cite{hopf} that $\ind(s)=\chi(M)$ is the Euler characteristic. 
Atiyah and Dupont showed that the index is always an invariant for $r\leq 3$. For $r=2,3$, they identified the index as follows:
\begin{table}[h]
\begin{equation*}
\begin{tabular}{cc}
\hline
 {$d \bmod 4$}&{$\ind(s)$}\\
\hline
 {$0$}&{$\chi\oplus \frac{1}{2}(\sigma + \chi) \in \Z \oplus \Z/a_r $} \\

 {$1$ }&{$\chi_\R \in \Z/2 $}\\
 
 {$2$ }&{$\chi \in \Z$}\\
 
 {$3$ }&{$0$}\\
\hline
\end{tabular}
\end{equation*}
\caption{$\ind(s) $ for $r=3$.}\label{indentifi}
\end{table}

Here $a_2=2$ and $a_3=4$. Moreover, $\chi_\R(M) \in \Z/2$ is the real Kervaire semi-characteristic given by
\begin{equation}\label{chi2}
\sum_k \dim_\R H^{2k}(M;\R) \mod 2,
\end{equation}
and $\sigma(M)$ is the signature of the quadratic form on $H^{\frac{d}{2}}(M;\R)$ defined by Poincar\'e duality.

The idea of Atiyah and Dupont in \cite{AD} and \cite{dupont} is to define a homomorphism 
\begin{equation*}
\tilde{\theta}^t_r: \pi_{d-1}(V_{d,r}) \to KR^t(tH_r) 
\end{equation*}
such that $\tilde{\theta}^t_r(\ind(s))$ is an invariant of $M$. Hence injectivity of $\tilde{\theta}_r^t$ would imply that also $\ind(s)$ were an invariant. They prove that $\tilde{\theta}^t_r$ is indeed injective for $r\leq 3$.

More precisely, they define a characteristic class 
\begin{equation*}
\tilde{\alpha}_{d,r}^t(E,s) \in KR^t(iE_{\mid X- Y}\times tH_r)
\end{equation*}
 for a vector bundle $E \to X$ with $r$ sections $s=\{s_1,\dots ,s_r\}$ given on a subcomplex $Y \subseteq X$. Here $KR$-theory is $K$-theory of spaces with involution, and $iE$ denotes $E$ with the involution given in each fiber by $x \mapsto -x$. Moreover, $tH_r$ denotes the sum of $t$ copies of the Hopf bundle over the projective space $\R P^{r-1}$ where $t$ is any number such that $4$ divides $ d+t$. 

For $X=M$ an oriented manifold with boundary $Y = \partial M$, the image $\tilde{\theta}^t_r(\ind(s))$ is the index $\ind(\tilde{\alpha}_{d,r}^t(TM,s))$ of this characteristic class. This depends only on the sections restricted to the boundary. In particular, it is the desired global invariant if the boundary is empty. The Atiyah--Singer index theorem applies to give the identification of this invariant displayed in Table \ref{indentifi}.

Following Atiyah and Dupont closely,  we set up similar invariants in Section~\ref{constr}. We replace $KR$-theory by another generalized cohomology theory defined by certain spectra $MT(d,r)$. These spectra, which go back to an old
unpublished note by the second author, are relative versions of the spectra $MT(d)$ studied in \cite{GMTW} and used by these authors in their well-known work on mapping class groups. 

Again we define characteristic classes, global invariants, and a map $\theta ^r$. At the end of the section, we show that $\tilde{\theta}^t_r$ factors as $\tilde{\theta}^t_r = \Psi \circ \theta^r$ for some map $\Psi$. Thus our invariants are refinements of the Atiyah--Dupont invariants, and we shall see in Theorem \ref{beta4} that they do carry strictly more information. The downside is that they are much harder to compute.

The computation of the homotopy groups $\pi_q(MT(d,r))$ for low values of $r$ and $q$ is the topic of Section \ref{cal}. These homotopy groups have an interesting periodicity property which is also useful for computations. We prove this in Section \ref{periodicity}. The results allow us to derive the injectivity results about $\theta^r $ in Theorem \ref{injtr}. 

For closed manifolds, our invariants are actually Reinhart cobordism invariants, see \cite{reinhart}. The relation to cobordism theory is studied further in Section~\ref{calcabs}. The computations of Section \ref{cal} allow us to perform low-dimensional computations of the homotopy groups of $MT(d)$. As a corollary, we obtain an identification of the global invariant for $r=4$.

We summarize our main results in the following:
\begin{thm}\label{results}
Assume $r<\frac{d}{2}$. If $M$ is oriented of even dimension $d$ and $r=4,5$ or $6$, then the index is an invariant of $M$. In this case, 
\begin{equation*}
\ind(s)=\chi(M)\in \Z 
\end{equation*}
for $d \equiv 2 \mod 4$, and for $d \equiv 0 \mod 4$,
\begin{equation*}
\ind(s)=\chi(M)\oplus \frac{1}{2}(\chi(M) +\sigma(M)) \in \Z\oplus \Z/8. 
\end{equation*}
If $M$ has a spin structure, the index is an invariant for $r\leq 6$ and, if $d$ is even, also for $r=7$.
\end{thm}
This follows from Theorem \ref{theta-inj}, \ref{injtr}, and  \ref{beta4} and Corollary \ref{spininj} below.

\section{Construction of the invariants}\label{constr} 
We begin this section by introducing the spectra $MT(d,r)$ in Section \ref{def1}. The characteristic class is constructed in Section \ref{intinv}, and in Section \ref{global}, we define the global invariant and give a geometric interpretation in the special case of closed manifolds. Then in Section \ref{relind}, we construct the map $\theta^r$ and show how it relates the index to the global invariants. We also derive some first injectivity results. Finally in Section \ref{relKR}, we relate the constructions to the ones of Atiyah and Dupont.
%
%

%

\subsection{A suitable spectrum}\label{def1}
Let $G(d,n)$ denote the Grassmannian consisting of $d$-dimensional subspaces of $\R^{d+n}$. Let $U_{d,n} \to G(d,n)$ be the vector bundle with fiber over a plane in $G(d,n)$ consisting of all points in that plane and with $n$-dimensional orthogonal complement $U_{d,n}^\perp$. Let $BO(d) = \varinjlim_{n} G(d,n)$ be the classifying space for $O(d)$ and $U_d = \varinjlim_n U_{d,n}$ the universal bundle.

%

Define $MTO(d)$ to be the spectrum with $n$th space
\begin{equation*}
MTO(d)_{n} = \Th(U_{d,n}^\perp )
\end{equation*}
where $\Th(\cdot )$ is the Thom space. The spectrum map $\Sigma MTO(d)_{n} \to MTO(d)_{n+1}$ is induced by the inclusion $G(d,n) \to G(d,n+1)$ taking a $d$-plane $P \subseteq \R^{n+d}$ to $0 \oplus P\subseteq \R^{1+n+d}$ as follows: The restriction of $U_{d,n+1}^\perp $ to $G(d,n)$ is $\R \oplus U_{d,n}^\perp$, defining an inclusion 
\begin{equation}\label{despe}
\Sigma \Th (U_{d,n}^\perp) = \Th (\R \oplus U_{d,n}^\perp ) \to \Th (U_{d,n+1}^\perp ).
\end{equation}

Under the inclusion $G(d-r,n) \to G(d,n)$ mapping a $(d-r)$-plane $P$ to $P \oplus \R^r$ in $\R^{n+d-r} \oplus \R^r$, the pullback of $U_{d,n}^\perp$ is $U_{d-r,n}^\perp $. This yields a map 
\begin{equation}\label{cof}
MTO(d-r)_n \to MTO(d)_n
\end{equation}
commuting with \eqref{despe} and thus defining a map of spectra. 
\begin{defi}
Denote by $MTO(d,r)$ the cofiber of \eqref{cof}, i.e.\ the spectrum with $n$th space
\begin{equation*}
MTO(d,r)_{n} = \Th(U_{d,n}^\perp ) / \Th(U_{d-r,n}^\perp ). 
\end{equation*}
\end{defi}

\begin{prop}
$MTO(d,r)$ is $(d-r)$-connected of finite type. 
\end{prop}

\begin{proof}
With the cell structures given in \cite{milnors}, $G(d,n)$ and $G(d,n+1)$ share the same $n$-skeleton, and since they are compact, all skeleta are finite. This shows the finite type. Also, $G(d-r,n)$ and $G(d,n)$ share the same $(d-r)$-skeleton for $n$ large, so the quotient has no cells in dimensions less than or equal to $d-r$.
\end{proof}

\begin{rem}
We may replace the Grassmann manifolds in the above construction with oriented Grassmannians. This yields an oriented version of the spectrum, which we denote by $MTSO(d,r)$. 

Another possibility is to look at a spin version of the spectrum. The simply connected double cover $Spin(d) \to SO(d)$ induces a fibration of classifying spaces $p : BSpin(d) \to BSO(d)$. We get a filtration of $BSpin(d)$ by the subspaces $ p^{-1}(G(d,n))$. The universal bundle pulls back to a bundle $p^*U_{d,n} \to p^{-1}(G(d,n))$. The construction also works in this case, and the resulting spectra will be denoted $MTSpin(d,r)$. 

Most constructions in the following work for all three spectra. In this case, we simply write $MT(d,r)$ for the spectrum and $G(d,n)$ for the corresponding filtration of the classifying space $B(d)$. 
\end{rem}

We now give a homotopy equivalent, and for some purposes more convenient, construction of the spectrum.
For a $d$-dimensional vector bundle $E \to X$ with inner product, there is a  fiber bundle 
\begin{equation*}
V_{d,r} \to  V_r(E) \to X.
\end{equation*}
The fiber over $x \in X$ is the set of ordered $r$-tuples of orthonormal vetors in the vector space $E_x$.

Similarly, there is a bundle $W_r(E)\to X$ with fiber the cone on $V_{d,r}$. A point in the fiber over $x \in X$ may be interpreted as an $r$-tuple of orthogonal vectors $v_1,\dots ,v_r$ in $E_x$ of equal length $0\leq |v_1|=\dotsm =|v_r|\leq 1$. 

There is an inclusion $\eta^r :G(d-r,n) \to V_r(U_{d,n})$ mapping the
$(d-r)$-plane $P$ to $P \oplus \R^r\subseteq \R^{n+d-r}\oplus \R^r$ with the $r$
standard basis vectors in $0\oplus \R^r$ as the $r$-frame. 
This can be extended to a section $\eta :G(d,n) \to
W_r(U_{d,n})$ such that the following diagram commutes
\begin{equation}\label{bundter}
\vcenter{\xymatrix{
{V_r(U_{d,n})} \ar[r]^{}&{W_r(U_{d,n})} \\
{{G}(d-r,n)}\ar[u]^{\eta^r }\ar[r]^{}&{G(d,n)} \ar[u]^{\eta }.}}
\end{equation} 


The projections $p_{W_r}: W_r(U_{d,n}^\perp) \to G(d,n)$ and  $p_{V_r}: W_r(U_{d,n}^\perp) \to G(d,n)$ induce a commutative diagram of bundle maps over the diagram \eqref{bundter}
\begin{equation*}
\xymatrix{{p_{V_r}^*U_{d,n}^\perp}\ar[r]&{p_{W_r}^*U_{d,n}^\perp}\\
{U_{d-r,n}^\perp}\ar[r]\ar[u]^{\eta^r}&{U_{d,n}^\perp.}\ar[u]^{\eta }
}
\end{equation*}

The Thom spaces $\Th(p_{W_r}^*U_{d,n}^\perp)$ and $\Th(p_{V_r}^*U_{d,n}^\perp)$ give rise to the spectra $MT(d)_{W_r}$ and $MT(d)_{V_r}$, respectively. 
Let $MTV(d,r)$ denote the cofiber of the inclusion $MT(d)_{V_r} \to MT(d)_{W_r}$.

\begin{thm} \label{altspek}
In the map of cofibration sequences defined by $\eta $,
\begin{equation}\label{cometaMT}
\vcenter{\xymatrix{{MT(d-r)}\ar[r]\ar[d]^{\eta^r}&{MT(d)}\ar[r]\ar[d]^{\eta}&{MT(d,r)}\ar[d]^{\bar{\eta}}\\
{MT(d-r)_{V_r}}\ar[r]&{MT(d)_{W_r}}\ar[r]&{MTV(d,r),}
}}
\end{equation}
all vertical maps are homotopy equivalences. 
\end{thm}

\begin{proof}
The section $\eta : G(d,n) \to W_r(U_{d,n})$ is a homotopy inverse of $p_{W_r}$, so it clearly induces a homotopy equivalence $MT(d) \to MT(d)_{W_r}$.

Consider the fibration
\begin{equation*}
G(d-r,n) \xrightarrow{\eta^r } V_r(U_{d,n}) \to V_{n+d,r}.
\end{equation*}
A point in $V_r(U_{d,n})$ is a plane $P \subseteq \R^{n+d}$ with $r$
orthonormal vectors. The last map just forgets $P$. Since $V_{n+d,r}$ is $(n+d-r-1)$-connected, the pair
$(V_r(U_{d,n}),G(d-r,n))$
is also $(n+d-r-1)$-connected. 

It follows that the pair $(\Th(p_{V_r}^*U_{d,n}^\perp ), \Th(U_{d-r,n}^\perp) )$ is $(2n+d-r-1)$-connected. 
Letting $n$ tend to infinity, we get isomorphisms of homotopy groups
\begin{equation*}
\eta^r_*: \pi_*(MT(d-r)) \to \pi_*(MT(d)_{V_r}).
\end{equation*}
Thus $\eta^r$ is a homotopy equivalence of spectra.

%

From the long exact sequences of homotopy groups for cofibrations of spectra applied to \eqref{cometaMT}, we see that also $\bar{\eta}$ induces an isomorphism on homotopy groups and thus is a homotopy equivalence.
\end{proof}
%
%
%
Given two spectra $X$ and $Y$, let $[X,Y]$ denote the abelian group of homotopy classes of spectrum maps $f:X \to Y$.
\begin{cor} \label{isor}
If $X$ is any spectrum, there is an isomorphism
\begin{equation*}
\bar{\eta}_* : [X,MT(d,r)] \to [X,MTV(d,r)].
\end{equation*} 
\end{cor} 

For $r=1$, ${MTV}(d,1)$ has a simple description, namely,
\begin{equation*}
 \Th (p_{W_1}^*  U_{d,n}^\perp)/\Th (p_{V_1}^* U_{d,n}^\perp )
 \cong \Th (U_{d,n}^\perp\oplus U_{d,n})
 = \Th({G}(d,n) \times \R^{n+d}).
\end{equation*}
From this we derive the following properties of $MT(d,1)$:
\begin{prop}\label{proj}
There is an isomorphism
\begin{equation*}
\pi_{q}(MTV(d,1)) \cong \pi_{q}^s(S^{d}) \oplus \pi^s_{q-d}(B(d)).
\end{equation*}
The map induced by
\begin{equation*}
c: \Th ({G}(d,n) \times \R^{n+d}) \to \Th (pt\times \R^{n+d}) 
\end{equation*}
is the projection onto the first direct summand. In particular, $c$ induces an isomorphism 
\begin{equation*}
c_* : \pi_{q}(MTO(d,1)) \to \pi_{q}^s(S^{d})
\end{equation*}
for $q\leq d$. 

For $MTSO(d,1)$, the last statement holds for $q\leq d+1$, and for $MTSpin(d,1)$, it is true for $q\leq d+3$.
\end{prop}

\begin{proof}
The first two claims follow  because $\Th (G(d,n) \times \R^{n+d})$ is homotopy e\-qui\-va\-lent to $S^{d+n}\vee \Sigma^{n+d}G(d,n) $, and 
under this equivalence, $c$ corresponds to the map that collapses $\Sigma^{n+d}G(d,n)$. 

Since $BO(d)$ is connected, $BSO(d)$ is simply connected, and $BSpin(d)$ is 3-connected, $c_*$ is an isomorphism in the dimensions claimed.
\end{proof}
 

\subsection{The characteristic class}\label{intinv}
We are now ready to introduce the characteristic classes promised in the introduction. The constructions work for both $MTO$, $MTSO$, and $MTSpin$.
In the latter two cases, we implicitly assume that all bundles involved have an orientation or a spin structure, respectively.

Suppose we are given a $d$-dimensional vector bundle $E \to X$ over a compact
$q$-dimensional CW complex $X$ and $r$ independent sections over a subcomplex $Y \subseteq X$. 
Choose a classifying map $\xi: X \to G(d,n)$ and an isomorphism $E \cong \xi^*U_{d,n}$. This defines an inner product on $E$. After applying the Gram--Schmidt process, the sections define a map $Y\to V_r(E)$. 
Extend this to a map
\begin{equation}\label{sectionV}
s: (X,Y)\to (W_r(E),V_r(E)).
\end{equation}
The bundle map $E \to U_{d,n}$ induces a map
\begin{equation}\label{classV}
(W_r(E),V_r(E)) \to (W_r(U_{d,n}),V_r(U_{d,n})).
\end{equation}
Let $N\cong \xi^*U_{d,n}^\perp$ be an $n$-dimensional complement of $E$ and let $p_{ W_r(E)} : W_r(E) \to X$ and $p_{ V_r(E)} : V_r(E) \to X$ be the projections.  
 Then there are bundle maps covering \eqref{sectionV} and \eqref{classV}
\begin{equation}\label{definv}
(N,N_{\mid Y}) \to (p^*_{W_r(E)}N,p^*_{V_r(E)}N) \to (p_{W_r}^*U_{d,n}^\perp, p_{V_r}^*U_{d,n}^\perp),
\end{equation}
inducing a map of Thom spaces
\begin{equation}\label{charac} 
(\Th (N),\Th (N_{\mid Y})) \to (MT(d)_{W_r,n},MT(d)_{V_r,n}).
\end{equation}

\begin{defi}
The map \eqref{charac} defines a characteristic class
\begin{equation*}
\alpha^r(E,s) \in  [\Sigma^{\infty -n} \Th (N)/\Th (N_{\mid Y}),MTV(d,r)] \cong MT_{d,r}^n(\Th (N),\Th (N_{\mid Y})).
\end{equation*}
Here $MT_{d,r}^*(-)$ is the generalized cohomology theory with coefficients in $MT(d,r)$ and the isomorphism is given by Corollary \ref{isor}.
\end{defi}

It is easy to check that
$\alpha^r(E,s)$ is independent of the choices made and only depends on $s_{\mid Y}$ up to homotopy of maps $Y\to V_r(E)$.

The given section $s: Y \to V_r(E)$ defines an isomorphism $E_{\mid Y} \cong E' \oplus \R^{r}$ for some $(d-r)$-dimensional bundle $E'$. Choose a classifying map $\xi_{\mid Y} : Y \to G(d-r,n)$ for $E'$. This extends to a classifying map $\xi : X \to G(d,n)$ for $E$. To see this, choose any classifying map $\xi'$ for $E$. Then $\xi'_{\mid Y}$ is homotopic to $\xi_{\mid Y} $, and the homotopy extension property yields the desired map. 

For such a classifying map, 
\begin{equation}\label{seccom}
\vcenter{\xymatrix{{V_r(E)}\ar[r]^-{\bar{\xi}}&{V_r(U_{d,n})}\\
{Y}\ar[r]^-{\xi}\ar[u]^-{s}&{G(d-r,n)}\ar[u]^-{\eta^r}
}}
\end{equation}
commutes. This yields an equivalent definition of the characteristic class.
\begin{prop}
For a classifying map $\xi:X \to G(d,n)$ such that \eqref{seccom} commutes up to homotopy, $\alpha^r(E,s)$ is represented in $MT_{d,r}^n(\Th(N),\Th(N_{\mid Y}))$ by the map 
\begin{equation*}
(\Th(N),\Th(N_{\mid Y}))\to (\Th(U^\perp_{d,n}),\Th(U^\perp_{d-r,n}))
\end{equation*}
induced by $\xi : (X,Y) \to (G(d,n),G(d-r,n))$.
\end{prop}

The following desirable properties of $\alpha^r$ are immediate from the definition:

\begin{prop}
If $s$ extends to $r$ independent sections on all of $X$, then $\alpha^r(E,s)$ is zero.
The characteristic class is natural with respect to maps $f: (X,Y) \to (X',Y')$ in the sense that $f^*(\alpha^r(E,s)) = \alpha^r(f^*E,f^*s)$ for any bundle $E \to X'$.
\end{prop}

\subsection{The global invariant}\label{global}
We now specialize to the case where $E=TM$ is the tangent bundle of a smooth manifold $M$ and $Y = \partial M$ is the boundary.
The fundamental class  $[M,\partial M]$ in $\pi_{n+d}^s(\Th(N)/\Th(N_{\mid \partial M}))$ is defined as follows. Embed $\partial M$ in $\R^{n+d-1}$, extend this to an embedding of a collar $\partial M \times [0,1]$ in $\R^{n+d -1} \times [0,1]$, and extend further to an embedding of $M $ in $\R^{n+d}$. Identify $N$ with a tubular neighbourhood $N_\varepsilon \subseteq \R^{n+d}$ of $M$. Then the fundamental class is represented by the Pontryagin--Thom map $S^{n+d} \to \Th(N)/\Th(N_{\mid \partial M})$ that collapses everything not in the interior of $N_\epsilon$ to a point.

\begin{defi}
If $M$ is a compact manifold with boundary $\partial M$ and a given section $s: \partial M \to V_r(TM)_{\mid \partial M}$, then the evaluation of $\alpha^r(TM,s)$ on the fundamental class defines an invariant depending only on the homotopy class of $s$
\begin{equation*}
\beta^r(M,s) = \langle \alpha^r(TM,s), [M, \partial M] \rangle \in MT_{d,r}^{-d}(S^0) \cong \pi_d(MT(d,r)).
\end{equation*}
This class may be represented by the composition of the Pontryagin--Thom map with \eqref{charac}.
\end{defi}

In this paper, we are mainly concerned with the case where $M$ is a closed manifold, i.e.\ $\partial M$ is empty. Since there is no dependence on $s$, $\beta^r(M,s)$ is an invariant of $M$ that vanishes if $M$ allows $r$ independent sections.
\begin{defi}
Let $M$ be a closed manifold.
\begin{itemize}
\item[(i)] Let $\beta^r(M) = \langle \alpha^r(TM),[M]\rangle \in \pi_d(MT(d,r))$ denote the global invariant for $M$.
\item[(ii)]Let $\beta(M) \in \pi_d(MT(d))$ be the class represented by the Pontryagin--Thom collapse followed by the classifying map
\begin{equation*}
S^{n+d} \to \Th(N) \to MT(d)_n.
\end{equation*}
\end{itemize}
\end{defi}
The cofibration sequence $MT(d-r) \to MT(d)\xrightarrow{j_r} MT(d,r)$ induces an exact sequence
\begin{equation}\label{jrdef}
\pi_d(MT(d-r)) \to \pi_d(MT(d))\xrightarrow{j_{r*}} \pi_d(MT(d,r)).
\end{equation}
Note that $\beta^r(M) = j_{r*}(\beta(M))$, so $\beta(M)$ is in some sense the universal invariant. 

Two closed $d$-dimensional manifolds $M$ and $N$ are said to be Reinhart cobordant if there is a compact manifold $W$ with boundary the disjoint union of $M$ and $N$  allowing a nowhere vanishing tangent vector field which is inward normal at $M$ and outward normal at $N$. The group of equivalence classes was determined in \cite{reinhart}.
It is well-known, see e.g.\ \cite{ab}, that $\pi_d(MT(d))$ is the Reinhart cobordism group and $\beta(M)$ represents the cobordism class of $M$.

The following geometric interpretation of the global invariant is shown in~\cite{ab}, Corollary 4.6:
\begin{prop}\label{betared} 
$\beta^r $ is a Reinhart cobordism invariant and $\beta^r(M)=0$ if and only if $M$ is Reinhart cobordant to a manifold allowing $r$ independent vector fields.
\end{prop}

\subsection{The local situation}\label{relind}
Look at the bundle $(D^q \times \R^d, S^{q-1} \times \R^d)$. Given a section $s:S^{q-1} \to V_{d,r}$, the above construction of the characteristic class $ \alpha^r(D^q \times \R^d,s)$ depends only on the homotopy class of $s$, so we can make the following definition:
\begin{defi}
Let $\theta^r$ be the map
\begin{align*}
\theta^r : \pi_{q-1}(V_{d,r}) \to MT_{d,r}^0( D^q, S^{q-1}) \cong \pi_q(MT(d,r))
\end{align*}
given by
\begin{equation*}
{\theta^r}([s]) = \alpha^r(D^q \times \R^d,s).  
\end{equation*}
\end{defi}

\begin{prop} \label{faktor}
 The map $\theta^r$ factors as the composition
\begin{equation*}
\pi_{q-1}(V_{d,r}) \to \pi^s_{q}(\Sigma V_{d,r}) \xrightarrow{f_{\theta *}} \pi_{q}({MTV(d,r)})
\end{equation*}
where the first map is the map into the direct limit. This is an isomorphism for $q \leq 2(d-r)$. The second map is induced by the map of Thom spaces over the inclusion of  a fiber 
\begin{equation*}
{f_{\theta}} : (W_{r}(\R^d\times pt),V_r(\R^d \times pt)) \to (W_r(U_{d,n}),V_r(U_{d,n})). 
\end{equation*}
In particular, $\theta ^r$ is a homomorphism.
\end{prop}

\begin{proof}
Let $s : S^{q-1} \to V_{d,r}$ be given and extend it to $s:D^q \to V_{d,r}$ by letting $s(x) = |x|s\big( \frac{x}{|x|}\big)$. The map in \eqref{definv} defining $\alpha^r(D^q\times \R^d,s)$ is given on base spaces by
\begin{equation*}
(D^q , S^{q-1})  \xrightarrow{id \times s} (D^q \times  W_r(\R^d) ,D^q \times  V_r(\R^d))  \xrightarrow{\xi} (W_r(U_{d,n}) , V_r(U_{d,n})).
\end{equation*}
The classifying map $\xi$ is constant, so this is the same as the composition
\begin{equation*}
(D^q , S^{q-1})  \xrightarrow{s} ( W_r(\R^d) , V_r(\R^d))  \xrightarrow{f_\theta} (W_r(U_{d,n}) , V_r(U_{d,n})).
\end{equation*}
On quotients, the first map is the suspension  $\Sigma s : \Sigma S^{q-1} \to \Sigma V_{d,r}$, while the second map is the inclusion of a fiber. Since $V_{d,r}$ is $(d-r-1)$-connected, the first map is an isomorphism for $q\leq 2(d-r)$ by the Freudenthal suspension theorem.
\end{proof}

\begin{prop} \label{ren}
Under the isomorphism of Proposition \ref{proj}, 
\begin{equation*}
\theta^1 : \pi_{q-1}(S^{d-1}) \to \pi_q(MT(d,1)) \cong \pi_{q}^s(S^{d})\oplus \pi_{q}^s(\Sigma^d B(d))
\end{equation*}
is the map into the first summand.
\end{prop}

\begin{proof}
Since $\theta^1$ factors as in Proposition \ref{faktor} and the composition 
\begin{equation*}
\Sigma^{n+1}(S^{d-1}) \xrightarrow{f_\theta } {MTV}(d,1)_n \cong \Th(G(d,n) \times \R^{d+n}) \to S^{d+n} 
\end{equation*}
is the identity, this follows from Proposition \ref{proj}.
\end{proof}

\begin{prop}\label{egenskaberr} 
For all $0 \leq k \leq r \leq d$, there is a cofibration sequence 
\begin{equation*}
MT(d-r+k,k) \to MT(d,r) \to MT(d,r-k).
\end{equation*}
The following diagram of long exact sequences commutes for $q < 2(d-r)$:
\begin{equation*}
\xymatrix@C=17pt{{}\ar[r] &{\pi_{q}(MT(d-r+k,k))}\ar[r]&{\pi_{q}(MT(d,r))}\ar[r]&{\pi_{q}(MT(d,r-k))}\ar[r]&{}\\
{}\ar[r]&{\pi_{q-1}(V_{d-r+k,k})}\ar[u]^{\theta^k} \ar[r]&{\pi_{q-1}(V_{d,r})}\ar[u]^{\theta^r} \ar[r]&{\pi_{q-1}(V_{d,r-k})} \ar[u]^{\theta^{r-k}} \ar[r]&{.}
}
\end{equation*}
The lower row is the long exact sequence for the fibration
\begin{equation*}
V_{d-r+k,k}\to V_{d,r} \to V_{d,r-k}
\end{equation*}
where the first map adds the last $r-k$ standard basis vectors to a $k$-frame as the last vectors in the frame and the second map forgets the first $k$ vectors of an $r$-frame.
\end{prop}

\begin{proof}
The cofibration is obvious from the definitions.

By Proposition \ref{faktor}, the groups in the lower row may be replaced  by the stable homotopy groups when $q \leq 2(d-r)$. The map $\bar{\eta}$ from Theorem \ref{altspek} defines a map of pairs
\begin{equation*}
(MT(d,r),MT(d-r+k,k)) \to ({MTV}(d,r),{MTV}(d-r+k,k)).
\end{equation*}
To see this, it is enough to check that the diagram
\begin{equation}\label{etacom}
\vcenter{\xymatrix{{(G(d-r+k,n),G(d-r,n))} \ar[r]^-{\eta } \ar[d] &{(W_k(U_{d-r+k,n}),V_k(U_{d-r+k,n}))} \ar[d]\\
{(G(d,n),G(d-r,n))}\ar[r]^-{\eta }&{(W_r(U_{d,n}),V_r(U_{d,n}))}
}}
\end{equation}
commutes up to homotopy. This is straightforward. 
Thus the upper row may be replaced by the long exact sequence for the pair $({MTV}(d,r),{MTV}(d-r+k,k))$.

The following map of long exact sequences is induced by $f_\theta$:
\begin{equation*}
\xymatrix@C=12pt{{\pi_{q}({MTV}(d-r+k,k))}\ar[r]&{\pi_{q}({MTV}(d,r))}\ar[r]&{\pi_{q}({MTV}(d,r)/{MTV}(d-r+k,k))}\\
{\pi_{q}^s(\Sigma V_{d-r+k,k})}\ar[u]^{f_{\theta *}} \ar[r]&{\pi_{q}^s(\Sigma V_{d,r})}\ar[u]^{f_{\theta *}} \ar[r]&{\pi_{q}^s(\Sigma V_{d,r}/\Sigma V_{d-r+k,k}).} \ar[u] 
}
\end{equation*}

By the above, it is enough to see that the last vertical map is actually $\theta^{r-k}$ under the isomorphism for $q < 2(d-r)$,
\begin{equation*}
\pi_{q}^s(\Sigma V_{d,r}/\Sigma V_{d-r+k,k}) \to \pi_{q}^s(\Sigma V_{d,r-k}). 
\end{equation*}
Again this follows by checking the definitions.
%
\end{proof}

We now return to the vector field problem.
Consider a compact connected oriented $d$-dimensional manifold $M$ with tangent bundle $TM$ and a given section $s:M\to W_r(M)$ with only finitely many singularities $x_1,\dots ,x_m$ in the interior of $M$. Choose small disjoint disks $D_j$ around each $x_j$.
Then the index was defined in the introduction by
\begin{equation*}
\ind(s) = \sum_{j=1}^m [s_{\mid \partial D_j}] \in \pi_{d-1}(V_{d,r}).
\end{equation*}
The map $\theta^r $ relates the index to the global invariant:
\begin{thm}\label{beta=ind}
The following formula holds in $\pi_d(MT(d,r))$:
\begin{equation*}
\beta^r(M,s) = \sum_{j=1}^m \theta^r([s_{\mid \partial D_j}]) = \theta^r(\ind(s)).
\end{equation*}
In particular when $M$ is closed, the right hand side is independent of $s$.
\end{thm}

\begin{proof}
 We may assume that the sections are orthonormal on $Y = M - \bigcup_j \indre (D_j)$. Then the map \eqref{definv} defining $\alpha^r(TM,s_{\mid \partial M})$ factors as
\begin{equation*}
\Th (N)/\Th(N_{\mid \partial M}) \to \Th (N)/\Th (N_{\mid Y}) = \bigvee_j S^{n+d}_j \to MT(d,r)_{n}. 
\end{equation*}
On each $S^{n+d}_j$, the last map is $f_\theta \circ \Sigma^{n+1}s_{\mid \partial D_j}$. The Pontryagin--Thom map
\begin{equation*}
S^{n+d} \to \Th(N)/ \Th(N_{\mid \partial M}) \to \Th (N)/\Th (N_{\mid Y}) = \bigvee_j S^{n+d}_j
\end{equation*}
is just the pinching map, so the composition is the sum of the $\theta^r([s_{\mid \partial D_j}])$. 
\end{proof}

\begin{rem}
For $r=1$ and $M$ closed, this is a classical result. It is shown in \cite{becker} that the degree of the composition
\begin{equation*}
S^{n+d} \to \Th(N) \to \Th(M \times \R^{n+d}) \to S^{n+d}
\end{equation*}
equals the Euler characteristic $\chi(M$), which is the index when $r=1$. This is also true if $M$ is not orientable. In the construction of $\beta^1$, we have just factored the last map through $\Th(G(d,n)\times \R^{n+d})$.
\end{rem}

For a closed manifold $M$, Theorem \ref{beta=ind} shows that $\theta^r(\ind(s))$ depends only on $M$. In particular, injectivity of $\theta^r$ implies that also $\ind(s)$ is an invariant. The following injectivity results are immediate from the definitions: 

%

\begin{thm}\label{theta-inj}
$\theta^1 $ is injective for all $q < 2(d-1)$. Assume $q < 2(d-r) - 1$.
Then 
\begin{equation*}
\theta^r : \pi_{q-1}(V_{d,r}) \to \pi_q(MTO(d,r))
\end{equation*}
is an isomorphism for  $q \leq d-r+1$. For $MTSO(d,r)$, it is an isomorphism for $q \leq d-r+2$, and for $MTSpin(d,r)$, it is an isomorphism for $q \leq d-r+4$ and injective for $q \leq d-r+6$.
\end{thm}

\begin{proof}
For $r=1$, injectivity follows from Proposition \ref{ren}, and the isomorphism statements follow from Proposition \ref{ren} and Proposition \ref{proj}. 

The case of a general $r$ follows from the diagram in Proposition \ref{egenskaberr} with $k=1$,
\begin{equation*}
\xymatrix@C=14pt{{}\ar[r]&{\pi_{q}(MT(d-r+1,1))}\ar[r]&{\pi_{q}(MT(d,r))}\ar[r]&{\pi_{q}(MT(d,r-1))} \ar[r]&{}\\
{}\ar[r]&{\pi_{q-1}(S^{d-r})}\ar[u]^{\theta^1} \ar[r]&{\pi_{q-1}(V_{d,r})}\ar[u]^{\theta^r} \ar[r]&{\pi_{q-1}(V_{d,r-1})} \ar[u]^{\theta^{r-1}} \ar[r] &{,}
}
\end{equation*}
using the result for $r=1$, induction on $r$, and the 5-lemma. 

 Since $\pi_{q-1}(S^{d-r}) = 0$ when $q=d-r+5$ or $q=d-r+6$, the above diagram becomes
\begin{equation}\label{spindia}
\vcenter{\xymatrix{{}&{\pi_{q}(MTSpin(d,r))}\ar[r]&{\pi_{q}(MTSpin(d,r-1))}\\
{0} \ar[r]&{\pi_{q-1}(V_{d,r})}\ar[u]^{\theta^r} \ar[r]&{\pi_{q-1}(V_{d,r-1})} \ar[u]^{\theta^{r-1}}
}}
\end{equation}
in the spin case. If $\theta^{r-1}$ is injective, then so is $\theta ^{r}$.
\end{proof}

\begin{rem}
A $d$-dimensional bundle $E\to X$ restricted to a $q$-cell is isomorphic to $D^q\times \R^d$. Given $s:\partial D\to V_{d,r}$, $[s]\in \pi_{q-1}(V_{d,r})$ is the obstruction to extending $s$ over the whole cell. Therefore, the maps
\begin{equation*}
\theta^r : \pi_{q-1}(V_{d,r})\to \pi_{q}(MT(d,r))
\end{equation*}
for $q<d$ may be useful when dealing with lower obstructions. A similar idea was succesfully applied in \cite{dupont} to determine the lower obstructions when $r\leq 3$. However, we have not investigated this further. 
\end{rem}

\subsection{Relation to the Atiyah-Dupont invariants}\label{relKR}

Let $E \to X$ be an oriented vector bundle with a section $s:Y \to V_r(E)$ given on $Y \subseteq X$. In \cite{AD}, Atiyah and Dupont defined a characteristic class
\begin{equation*}
\tilde{\alpha}^t_{d,r}(E,s) \in KR^t(iE_{\mid (X-Y)} \times tH_r).
\end{equation*}
This class is natural with respect to bundle maps. In particular, there is a universal class 
\begin{equation*}
\tilde{\alpha}^t_{d,r} \in \varprojlim_n KR^t(i  U_{d,n \mid (G(d,n)-G(d-r,n))}  \times tH_r)
\end{equation*}
such that the characteristic class for any other bundle is the pullback of this class by the classifying map. 

Let $N$ be an $n$-dimensional complement of $E$. We want to define a map 
\begin{equation}\label{KR-red}
\Psi : MTSO_{d,r}^{n-q}(\Th(N),\Th(N_{\mid Y})) \to KR^{t-q}(iE_{\mid (X-Y)}\times tH_r).
\end{equation}
For this, let
\begin{equation*}
f:\Sigma^q \Th (N)/\Th (N_{\mid Y}) \to MT(d,r)_n
\end{equation*}
be any map.  Let $i_{d,n} : U_{d,n} \to U_d$ be the inclusion. Then 
\begin{equation*}
i_{d,n}^*(\tilde{\alpha}^t_{d,r}) \in KR^t(iU_{d,n \mid (G(d,n)-G(d-r,n))} \times tH_r).
\end{equation*}

There is a Thom isomorphism
\begin{align*}
KR^{t}(iE_{\mid (X-Y)} \times tH_r) &\cong KR^t ((iE \oplus iN \oplus N)_{\mid (X-Y)}\times tH_r)\\
 &= KR^{t+n+d}(N_{\mid (X-Y)} \times tH_r),
\end{align*}
so we have a diagram where $\phi_1, \phi_2$ are Thom isomorphisms
\begin{equation*}
\xymatrix{ {KR^{t-q}(iE_{\mid (X-Y)} \times tH_r)} \ar[r]^-{\phi_2}& {KR^{t+n+d-q}(N_{\mid (X-Y)}\times tH_r)}\\
{KR^{t}(iU_{d,n \mid (G(d,n)-G(d-r,n))}\times tH_r)} \ar[r]^-{\phi_1}& {KR^{t+n+d}(U^\perp_{d,n \mid (G(d,n)-G(d-r,n))}\times tH_r).}\ar[u]^{f^*}
}
\end{equation*}
Now define 
\begin{equation*}
\Psi([f]) =  \phi_2^{-1}\circ f^* \circ \phi_1(i_{d,n}^*(\tilde{\alpha}^t_{d,r}))\in KR^{t-q}(iE_{\mid (X-Y)}\times tH_r). 
\end{equation*}
Obviously, if $f$ comes from  a classifying map $\xi : (X,Y) \to (G(d,n),G(d-r,n))$ for $E$, then by naturality of the Thom isomorphism, we get:
\begin{prop}
\begin{equation*}
\Psi(\alpha^r(E,s)) = \xi^*(i_{d,n}^*(\tilde{\alpha}^t_{d,r})) = \tilde{\alpha}^t_{d,r}(E,s).
\end{equation*}
\end{prop}
Thus $\alpha^r(E,s)$ is in some sense a refinement of the characteristic class defined by Atiyah and Dupont.
The special case $(X,Y)=(D^q,S^{q-1})$ yields: 
\begin{thm}\label{KRfak}
The map $\tilde{\theta}^t_r$ constructed in \cite{AD} factors as the composition
\begin{equation*}
\pi_{q-1}(V_{d,r}) \xrightarrow{\theta^r}  \pi_q({MTSO(d,r)}) \xrightarrow{\Psi } {KR^{t-q}(tH_r)}.
\end{equation*}
\end{thm}

In particular, our $\theta^r$ may throw away less information about the index than $\tilde{\theta}^t_r$ does.
The injectivity results for $\tilde{\theta}^t_r$ given in \cite{AD}, Proposition~5.6, imply:
\begin{cor}\label{KRinj}
$\theta^r : \pi_{q-1}(V_{d,r}) \to \pi_q(MTSO(d,r))$ is injective for $q\leq d-r+3$ and $d\geq r+3$. Moreover, $\theta^5 : \pi_{d-1}(V_{d,5}) \to \pi_{d}(MTSO(d,5))$ is injective when $8 \mid d$.
\end{cor}

\section{Low-dimensional calculations}\label{cal}
Having set up the invariants in the previous section, we now look for conditions under which the map
\begin{equation*}
\theta^r : \pi_{q-1}(V_{d,r}) \to \pi_{q} (MT(d,r))
\end{equation*}  
is injective. The groups $\pi_{q-1}(V_{d,r})$  have been calculated in \cite{paechter} for $r\leq 6$. The purpose of this section is to compute the homotopy groups $\pi_{q} (MT(d,r))$ and use this to show injectivity of $\theta^r$ for certain values of $q$, $d$, and $r$.

The computations are carried out by means of the Adams spectral sequence. Therefore we recall this and its basic properties we need in Section~\ref{as}. The input for the spectral sequence is the cohomology of $MT(d,r)$, so we determine this in Section~\ref{cohomology af MT}. In Section \ref{periodicity}, we prove a periodicity property for the spectra $MT(d,r)$ that will come in handy. The actual computations are performed in the oriented case in Section \ref{calcSO} and the unoriented case in Section \ref{calcO}. In the spin case, the cohomology considerations are a bit more complicated. We consider this case in Section~\ref{spin}.


\subsection{The Adams spectral sequence}\label{as}
If $X$ is a spectrum of finite type, $H^*(X;\Z/p) \cong \varprojlim_n \widetilde{H}^{*+n}(X_n;\Z/p)$ is a module over the mod $p$ Steenrod algebra $\mathcal{A}_p$ in the obvious way.
Based on this module structure, the Adams spectral sequence computes the $p$-primary part of $\pi_*(X)$, i.e.\ $\pi_*(X)$ divided out by torsion elements of order prime to $p$: 
\begin{thm}
For a connective spectrum of finite type, there is a natural spectral sequence $\{ E_k^{s,t}, d_k\} $ with differentials $d_k: E_k^{s,t} \to E_k^{s+r,t+r-1}$ such that:
\begin{itemize}
\item[(i)] $E_2^{s,t} = \Ext^{s,t}_{\mathcal{A}_p}(H^*(X;\Z/p), \Z/p)$.
\item[(ii)] There is a filtration 
\begin{align*}
\dotsm \subseteq F^{s+1,t+1} \subseteq F^{s,t}\subseteq \dotsm \subseteq F^{0,t-s}=\pi_{t-s}(X)
\end{align*}
 such that $E^{s,t}_\infty = F^{s,t}/F^{s+1,t+1}$.
\item[(iii)] $\bigcap_k F^{s+k,t+k}$ is the subgroup of $\pi_{t-s}(X)$ consisting of elements of finite order prime to $p$. 
\end{itemize}
\end{thm} 
See e.g.\ \cite{hatcher2} for the construction.

$\Ext^{s,t}_{\mathcal{A}_p}(H^*(X;\Z/p), \Z/p)$ is defined as follows. Take a resolution of $H^*(X;\Z/p)$, i.e.\ an exact sequence
\begin{equation*}
0 \leftarrow H^*(X;\Z/p) \leftarrow F_0 \leftarrow F_1 \leftarrow F_2 \leftarrow \cdots 
\end{equation*}
where each $F_s$ is a free $\mathcal{A}_p$-module. Then apply the functor $\Hom^t_{\mathcal{A}_p}(-,\Z/p)$, i.e.\ homomorphisms to the $\mathcal{A}_p$-module $\Z/p$ that lower degree by $t$. The homology of this dual complex at $F_s$ is $\Ext^{s,t}_{\mathcal{A}_p}(H^*(X;\Z/p), \Z/p)$.

Naturality means that a map $f: X \to Y$ induces a filtration preserving map $f^* : H^*(Y;\Z/p) \to H^*(X;\Z/p)$. The extension to a map of resolutions defines a map of spectral sequences. 

There is a pairing
\begin{equation*}
\Ext^{s,t}_{\mathcal{A}_p}(\Z/p,\Z/p) \otimes \Ext^{s',t'}_{\mathcal{A}_p}(H^*(X;\Z/p), \Z/p) \to \Ext^{s+s',t+t'}_{\mathcal{A}_p}(H^*(X;\Z/p), \Z/p)
\end{equation*}
defined algebraically using resolutions. This induces a product
\begin{equation*}
\bar{E}^{s,t}_k \otimes E^{s',t'}_k \to E_k^{s+s',t+t'} 
\end{equation*}
where $\bar{E}$ denotes the spectral sequence for the sphere spectrum. The product converges to the composition product
\begin{equation*} 
\pi_*^s(S^0) \otimes \pi_*(X) \to \pi_*(X).
\end{equation*}
That is, the composition product respects the filtrations given by the spectral sequences and the induced product on $E_\infty$ agrees with the one coming from the product on the spectral sequences. 
The differentials behave nicely on products:
\begin{equation*}
d_k(xy) = xd_k(y) + (-1)^{t-s}d_k(x)y.
\end{equation*}
See \cite{mccleary} for a more detailed explanation of this product.
 
We are mainly interested in the elements $h_0 \in \Ext^{1,1}_{\mathcal{A}_p}(\Z/p,\Z/p)$  corresponding to a degree $p$ map and $h_1\in \Ext^{1,2}_{\mathcal{A}_p}(\Z/2,\Z/2)$ corresponding to the Hopf map. If $x\in E^{s,t}_2$ represents some map in $\pi_{t-s}(X)$, $h_0x$ represents $p$ times this map, provided that it survives to $E_\infty$. Similarly, $h_1x$ corresponds to the map composed with the Hopf map. This way, the product will allow us to determine extensions.

\subsection{Cohomology of the spectrum}\label{cohomology af MT}

Since the input for the Adams spectral sequence is the cohomology structure of $MT(d,r)$, we first compute this.
We start out by recalling the cohomology of Grassmannians. A good reference for this is \cite{milnors}.

For a $d$-dimensional vector bundle $p:E \to X$, there is a Thom isomorphism 
\begin{equation*}
\phi: H^*(X;\Z/2) \to \widetilde{H}^{*+d}(\Th(E);\Z/2)
\end{equation*}
with Thom class $u \in H^{d}(\Th(E);\Z/2)$.

\begin{defi}\label{defsw}
Let $w_i(E)\in H^i(X;\Z/2)$ be the class $\phi^{-1}(\Sq^i(u))$. This is called the $i$th Stiefel--Whitney class of $E$.
\end{defi}
%
%

Let $w_i$ denote the $i$th Stiefel--Whitney class for the universal bundle $U_d \to BO(d)$, and
let $R[x_1,\dots ,x_k]$ denote the polynomial algebra over the ring $R$ on generators $x_i$ of degree~$i$. 

\begin{thm} \label{swcoho}
\begin{equation*}
\begin{split}
H^*(BO(d);\Z/2) &\cong \Z/2[w_1,\dots,w_d]\\
H^*(BSO(d);\Z/2) &\cong \Z/2[w_2,\dots,w_d].
\end{split}
\end{equation*}
The restriction $H^*(B(d);\Z/2)\to H^*(G(d,n);\Z/2)$ is surjective in both the oriented and unoriented case with kernel the ideal generated by the classes
\begin{equation*}
 \bar{w}_{n+1},\bar{w}_{n+2}, \dots 
\end{equation*}
characterized by $\bar{w}_0=1$ and
\begin{equation*}
\sum_{j} w_j \bar{w}_{i-j} = 1.
\end{equation*}
\end{thm}

%

\begin{thm} Let $F$ denote one of the fields $\Q$ and $\Z/p$ for $p$ an odd prime. Then
\begin{equation*}
H^*(BO(d);F) \cong F[p_1,\dots ,p_{[\frac{d}{2}]}].
\end{equation*}
Here $p_i \in H^{4i}(BO(d);F)$ are the Pontryagin classes for the universal bundle. For $x \in \R$, $[x]$ denotes the integer part of $x$.

When $d$ is odd, $BSO(d) \to BO(d)$ induces an isomorphism on cohomology with coefficients in $F$. For $d$ even,
\begin{equation*}
H^*(BSO(d);F) \cong F[p_1,\dots ,p_{\frac{d}{2}},e_d] /\langle e_d^2 - p_{\frac{d}{2}} \rangle
\end{equation*}
where the $p_i$ are the Pontryagin classes for the oriented universal bundle $U_d$ and $e_d$ is the Euler class in $H^d(BSO(d);F)$ for $U_d$, satisfying the relation $e_d^2 = p_{\frac{d}{2}}$.
\end{thm}
 
We are now ready to describe the cohomology of the spectra $MT(d,r)$. Unless otherwise specified, everything works for both the oriented and unoriented spectrum. 

We first consider cohomology with $\Z/2$ coefficients understood.
The group $H^k(MT(d))$ is given by $\varprojlim_n \widetilde{H}^{k+n}(\Th(U_{d,n}^\perp))$.
Each bundle $U^\perp_{d,n}$ has a Thom class $\bar{u}_n\in H^n(\Th (U^\perp_{d,n}))$,
 defining a stable Thom class $\bar{u}$ in $H^0(MT(d))$. There is a Thom isomorphism given by cup product with $\bar{u}$
\begin{equation*}
\phi : H^*(B(d)) \to H^*(MT(d)).
\end{equation*}
Hence $H^*(MT(d))$ is the free module over $H^*(B(d))$ generated by $\bar{u}$. 

\begin{thm}\label{fiw}
The map induced by the quotient
\begin{equation*}
H^*(MT(d,r)) \to H^*(MT(d))
\end{equation*}
is injective with image the $H^*(B(d))$-submodule generated by $\phi(w_{d-r+1}),\dots ,\phi(w_d)$. 
\end{thm}
\begin{proof}
There is a long exact sequence
\begin{equation*}
\dotsm \to H^*(BO(d),BO(d-r)) \to H^*(BO(d)) \to H^*(BO(d-r)) \to \dotsm.
\end{equation*}
By Theorem \ref{swcoho}
and naturality of the Stiefel--Whitney classes, the map  
\begin{equation*}
H^*(BO(d)) \to H^*(BO(d-r))
\end{equation*}
is the surjection 
\begin{equation*}
\Z /2[w_1,\dots ,w_d] \to \Z /2[w_1,\dots,w_{d-r}].
\end{equation*}
Thus $H^*(BO(d),BO(d-~r))$ is the exactly the kernel, i.e.\ the ideal generated by $w_{d-r+1},\dots,w_d$.


For any $n$, the Thom isomorphism yields a commutative diagram 
\begin{equation*}
\xymatrix@C=14pt{{}\ar[r]&{H^k(\Th(U_{d,n}^\perp ) ,\Th(U^\perp_{d-r,n}))}\ar[r]&{\widetilde{H}^k(\Th(U^\perp_{d,n}))}\ar[r]&{\widetilde{H}^k(\Th(U_{d-r,n}^\perp))}\ar[r]&{}\\
{}\ar[r]&{H^k(G(d,n),G(d-r,n))}\ar[u]^\cong \ar[r]&{H^k(G(d,n))}\ar[u]^\cong \ar[r]&{H^k(G(d-r,n))} \ar[u]^\cong
\ar[r]&{.}}
\end{equation*} 
Letting $n$ tend to infinity, the claim follows. The oriented case is similar.
\end{proof} 

The multiplication 
\begin{equation*}
H^*(B(d)) \otimes H^*(MT(d)) \to H^*(MT(d))
\end{equation*}
is induced by the diagonal $MT(d) \to B(d)_+\wedge MT(d)$ where $_+$ is disjoint union with a point. This takes $MT(d-r) $ to $ B(d)_+\wedge MT(d-r)$. Hence there is a diagonal $MT(d,r) \to B(d)_+\wedge MT(d,r)$ inducing the $H^*(B(d))$-module structure on $H^*(MT(d,r))$. The diagram 
\begin{equation*}
\xymatrix{{MT(d,r)}\ar[r]&{B(d)_+\wedge MT(d,r)}\\
{\Sigma^{\infty}\Sigma V_{d,r}}\ar[u]^{f_\theta}\ar[r]&{S^0 \wedge \Sigma^{\infty}\Sigma V_{d,r}}\ar[u]^{i\wedge f_\theta}
}
\end{equation*}
shows that $f_\theta^* : H^*(MT(d,r))\to H^*(\Sigma^\infty \Sigma V_{d,r})$ is a homomorphism of $H^*(B(d))$-modules
when $H^*(\Sigma^\infty \Sigma V_{d,r})$ is given the trivial module structure.

\begin{thm}  \label{cohom-cofiber}
Assume $k < 2(d-r)$. Then
\begin{align*}
\widetilde{H}^k (V_{d,r})=\begin{cases}
\Z /2 & \textrm{if $d-r \leq k \leq d-1$},\\
0 & \textrm{if $k < d-r $},\\
0 & \textrm{if $ d \leq k $.}
\end{cases} 
\end{align*}
The map 
\begin{equation*}
H^*(MT(d,r)) \xrightarrow{f_{\theta}^*} {H}^*(\Sigma^\infty \Sigma V_{d,r})
\end{equation*}
is the $H^*(B(d))$-module homomorphism that takes $\phi(w_k) \in H^k(MT(d,r))$ to the generator in ${H}^{k}( \Sigma V_{d,r})$. 
\end{thm}
\begin{proof}
 There is a $2(d-r)$-connected map $\R P^{d-1}/\R P^{d-r-1}\to V_{d,r}$, see \cite{james}. The first claim follows from this.

It remains to determine $f_\theta^*(\phi(w_k))$. For $r=1$, $f_\theta $ is the inclusion
\begin{equation*}
\Sigma^{\infty+d}S^{0} \to \Sigma^{\infty+d}B(d)_+.
\end{equation*}
This clearly induces an isomorphism on $H^d$, and the claim follows. For general $r$, the claim follows from the diagram
\begin{equation*}
\xymatrix{{}\ar[r]&{H^*(MT(d,1))}\ar[r]\ar[d]^{f_\theta^*}&{H^*(MT(d,r))}\ar[r]\ar[d]^{f_\theta^*}&{H^*(MT(d-1,r-1))}\ar[r]\ar[d]^{f_\theta^*}&{}\\
{}\ar[r]&{H^*(\Sigma V_{d,1})}\ar[r]&{H^*(\Sigma V_{d,r})}\ar[r]&{H^*(\Sigma V_{d-1,r-1})}\ar[r]&{}}
\end{equation*} 
by induction.
\end{proof}

Let $C_\theta $ denote the cofiber of the inclusion of spectra
\begin{equation} \label{cofiber}
f_\theta : \Sigma ^{\infty+1} V_{d,r} \xrightarrow{} MT(d,r).
\end{equation} 

\begin{cor}\label{cofib2}
The cofibration \eqref{cofiber} induces an injective map
\begin{equation*}
H^k(C_{\theta }) \to H^k(MT(d,r)) 
\end{equation*}
for $k\leq 2(d-r)+1$. The image is the $H^*(B(d))$-submodule
\begin{equation*}
H^{>0}(B(d))\cdot H^*(MT(d,r)) 
\end{equation*}
in dimensions $k\leq 2(d-r)$.
\end{cor}

\begin{proof}
This follows from the long exact sequence in cohomology for the cofibration~\eqref{cofiber} and Theorem \ref{cohom-cofiber}.
\end{proof}

For the remainder of this section, let $F$ be either $\Q$ or $\Z/p$ for $p$ an odd prime. For $MTSO(d)$, there is again a Thom isomorphism $\phi$ on cohomology with coefficients in $F$. This allows us to compute the cohomology groups.
\begin{thm}\label{Fcoeff}
In dimensions $*\leq 2(d-r)+1$, $H^*(MTSO(d,r);F)$ is isomorphic to the free $H^*(BSO(d);F)$-module on generators 
\begin{align*}
\phi(\delta (e_{d-r}))&\in H^{d-r+1}(MTSO(d,r);F)\\
 \phi(e_{d}') &\in H^{d}(MTSO(d,r);F). 
\end{align*}
Here $\delta : H^{d-r}(BSO(d-r);F)\to H^{d-r+1}(BSO(d),BSO(d-r);F)$ is the coboundary map.
Moreover, $ e_{d-r}\in H^{d-r}(BSO(d-r);F)$ is the Euler class and $e_d' $ maps to the Euler class $e_d \in H^d(BSO(d);F)$. These are zero exactly when $d-r$ and $d$, respectively, are odd.
\end{thm}

\begin{proof}
The claim follows from the long exact sequence
\begin{equation*}
\to H^*(BSO(d),BSO(d-r);F) \to H^*(BSO(d);F) \to H^*(BSO(d-r);F) \to 
\end{equation*}
since the Pontryagin classes $p_i \in H^{4i}(BSO(d);F)$ map to the corresponding classes in $H^{4i}(BSO(d-r);F)$, $e_d$ maps to the Euler class for $U_{d-r} \oplus \R^r$, which is zero, and $e_{d-r}$ is not in the image of $H^*(BSO(d);F)$. 
%
\end{proof}

\begin{thm}\label{cohom-cofiber2}
Assume that  $k<2(d-r)$.
\begin{align*}
\widetilde{H}^k (V_{d,r}; F)=\begin{cases}
F & \text{if $k = d-r$ and $d-r$ is even},\\
F & \text{if $k = d-1$ and $d$ is even},\\
0 & \text{otherwise}.
\end{cases}
\end{align*}
The map $f_{\theta}^* :H^*(MTSO(d,r);F) \to  H^*(\Sigma V_{d,r};F)$ is the $H^*(BSO(d);F)$-module homomorphism taking $\phi(\delta(e_{d-r}))$ to a generator of $ H^{d-r+1} (\Sigma V_{d,r};F)$ 
and $\phi(e_d')$ to a generator of $ H^{d} (\Sigma V_{d,r};F) $. 
\end{thm}

\begin{proof}
The proof goes more or less as in the $\Z/2$ case.
%
%
%
\end{proof}

\subsection{A periodicity map}\label{periodicity}
Let $a_r$ be the number given by the table 
\begin{equation}\label{RH}
\begin{tabular}{|c|c|c|c|c|c|c|c|c|}
\hline 
{$r$}&{1}&{2}&{3}&{4}&{5}&{6}&{7}&{8}\\
\hline
{$a_r$}&{1}&{2}&{4}&{4}&{8}&{8}&{8}&{8}\\
\hline
\end{tabular}
\end{equation}
for $r \leq 8$, and in general by $a_{r+8} = 16 a_r$. 
The number $a_{r+1}$ is the least integer such that $\R^{a_{r+1}}$ is a module over the Clifford algebra $Cl_{r}$. See \cite{lawson} for details. 
In \cite{adams1}, Adams related this number to the vector field problem as follows:
\begin{thm}\label{vfsp}
$S^{d-1}$ allows $r$ independent vector fields if and only if $d$ is divisible by $a_{r+1}$. 
\end{thm}
 If $a_{r+1} \mid d$, the vector fields may be constructed as follows: Let $V$ be a $d$-dimensional representation for $Cl_{r}$.
  This module restricts to a bilinear multiplication $\R^{r+1} \times V \to V$ such that if $e_0,\dots ,e_{r}$ is the standard basis for $ \R^{r+1}$, $e_0 $ acts as the identity and for a suitable inner product $\langle \cdot,\cdot \rangle$ on $V $, $\langle e_ix,e_jx \rangle =\delta_{ij}\langle x,x \rangle$ for all $x\in V$, see \cite{lawson}. 
In particular for $x \in S^{d-1}$, this means that $x=e_0x,\dots,e_rx$ are orthonormal, i.e.\ $e_1x,\dots,e_rx$ define the desired vector fields. 

We shall use this idea to construct a periodicity map.
Let $\R^{d+k} = \R^d \oplus \R^k$ and assume $a_r \mid k $. As above, there exists an orthogonal bilinear multiplication $\R^{r}\times \R^k \to \R^k$. This defines a map
\begin{equation} \label{Vmap}
f_0 : (W_{d,r}, V_{d,r}) \times  (D^k,S^{k-1}) \to (W_{d+k,r},V_{d+k,r})
\end{equation}
by
\begin{equation*}
(v_0,\dots ,v_{r-1},x) \mapsto (   \sqrt{1-|x|^2}v_0 + e_0x,\dots,  \sqrt{1-|x|^2}v_{r-1} + e_{r-1}x).
\end{equation*}

\begin{lem}\label{Vper}
The map $f_0: \Sigma^{k}\Sigma V_{d,r} \to \Sigma V_{d+k,r}$ in \eqref{Vmap} is a $(2(d-r)+k+1)$-equivalence. 
\end{lem}
 
\begin{proof}
The proof will proceed by induction on $r$. For $r=1$, the map
\begin{equation*}
(D^{d},S^{d-1}) \times (D^k, S^{k-1}) \to (D^{d+k},S^{d+k-1})
\end{equation*}
is given by
\begin{equation*}
(x,y) \mapsto (\sqrt{1-|x|^2}y+x),
\end{equation*}
which is a homeomorphism $S^d \wedge S^{k} \to S^{d+k}$.

More generally, look at the map $g: S^{d-r} \to V_{d,r}$ mapping $v_0 \in S^{d-r}$ to the frame $(v_0, u_1,\dots ,u_{r-1})$ where $u_1,\dots ,u_{r-1}$ are the first $r-1$ standard basis vectors in $\R^d$ and $v_0 \in \R^{d-r+1}$ is included in $\R^{d} = \R^{r-1} \oplus  \R^{d-r+1} $. The following diagram commutes up to homotopy
\begin{equation*}
\xymatrix{{\Sigma^{k}\Sigma S^{d-r}}\ar[r]^{\Sigma^{k}\Sigma g} \ar[d]^{f_0}&{\Sigma^{k}\Sigma V_{d,r}}\ar[d]^{f_0}\\
{\Sigma S^{d+k-r}}\ar[r]^{\Sigma g}& {\Sigma V_{d+k,r}.}
}
\end{equation*} 
A homotopy is given by applying to the $(i+1)$th vector the homotopy
\begin{equation}\label{HomV}
F(x,v_0,t)= \sqrt{(1-t^2)|x|^2+ (1-|x|^2)|v_0|}u_i + te_i x.
\end{equation}
Both horizontal maps in the diagram induce isomorphisms on cohomology with $\Z/2$ coefficients in dimension $d-r+k+1$ and so does the $f_0$ to the left, hence so must the $f_0$ to the right. Similarly, $f_0$ induces an isomorphism on cohomology with $\Q$ and $\Z/p$ coefficients in this dimension.

Consider the map $h:V_{d,r} \to V_{d,r-1}$ that forgets the last vector. This is an isomorphism on cohomology 
in dimension $d-r+1,\allowbreak \dots ,2(d-r)$. The diagram
\begin{equation*}
\xymatrix{{\Sigma^{k+1} V_{d,r}}\ar[r]^-{h} \ar[d]^{f_0}&{\Sigma^{k+1} V_{d,r-1}}\ar[d]^{f_0}\\
{\Sigma V_{d+k,r}}\ar[r]^-{h}& {\Sigma V_{d+k,r-1}}
}
\end{equation*} 
commutes. Thus the left vertical map induces an isomorphism on cohomology in dimensions $d+k-r+2,\allowbreak \dots ,2(d-r)+k + 1$, since the other three maps do so by induction. Similarly with $\Q$ and $\Z/p$ coefficients.
%

By the universal coefficient theorem for homology, $f_0$ also induces an isomorphism on homology with integer coefficients. So by the Whitehead theorem, the map is a $(2(d-r)+k+1)$-equivalence.
\end{proof}

Since $f_0$ is independent of the coordinates in $\R^d$, it extends to a map 
\begin{equation*}
f:  ( W_{r}(U_{d,n}), V_{r}(U_{d,n}))\times (D^k,S^{k-1}) \to (W_{r}(U_{d+k,n}),V_{r}(U_{d+k,n}))
\end{equation*}
given by 
\begin{equation*}
f(V,v_1,\dots,v_r,x) = (V \oplus \R^{k}, \sqrt{1-|x|^2}v_0 + e_0x,\dots ,\sqrt{1-|x|^2}v_{r-1} + e_{r-1}x).
\end{equation*}
This defines a map $f: MTO(d,r)\wedge S^k \to MTO(d+k,r)$. For $MTSO$ and $MTSpin$, $f$ is constructed similarly. 

\begin{thm}\label{perfrembr}
For $MTO$ and $MTSO$, the map $f$ induces an isomorphism of $H^*(B(d+k);\Z/2)$-modules
\begin{equation*}
H^{*+k}(MT(d+k,r);\Z/2) \to H^*(MT(d,r);\Z/2) 
\end{equation*}
in dimensions $*\leq 2(d-r+1)$. It takes the generators $\phi(w_{d+k-r+1}),\dots ,\phi(w_{d+k})$ to $\phi(w_{d-r+1}),\dots,\phi(w_{d})$. In the  $MTSO(d,r)$ case, $f$ is a $(2(d-r+1)+k)$-equivalence. 
\end{thm}

In the following, $\phi(w_i) \in H^i(MT(d,r);\Z/2)$ will just be denoted by $w_i$ for simplicity.

\begin{proof}
Consider the diagram 
\begin{equation}\label{ftheta}
\vcenter{\xymatrix{{\Sigma^{\infty +k+1}V_{d,r}} \ar[r]^-{f_\theta }  \ar[d]^{f_0}&{\Sigma^{k}MT(d,r)} \ar[r]^-{} \ar[d]^{f}&{\Sigma^{k} MT(d,r) \wedge B(d)} \ar[d]^{f \wedge i}\\
{\Sigma^{\infty +1} V_{d+k,r}}\ar[r]^-{f_\theta } &{MT(d+k,r)}\ar[r]^-{} &{MT(d+k,r) \wedge B(d+k).}
}}
\end{equation}
 By Theorem \ref{fiw}, $H^*(MT(d+k,r);\Z /2)$ is isomorphic to the free $H^*(B(d+k); \Z /2)$-module on generators $w_{d+k-r+1},\dots ,w_{d+k}$ in dimensions $* < 2(d+k-r+1)$ and $H^*(MT(d,r); \Z /2)$ is the free $H^*(B(d);\Z/2)$-module generated by $w_{d-r+1},\dots , w_{d}$ up to dimension $2(d-r+1)$. 
Moreover, $i^*:H^*(B(d+k);\Z/2) \to H^*(B(d);\Z/2)$ is an isomorphism up to dimension $d$, and $f^*$ respects the module structure by the right hand side of the diagram. Thus we just need to check that $f^*$ maps $w_{d+k-r+1},\dots ,w_{d+k}$ to $w_{d-r+1},\dots , w_{d}$. 

It obviously takes $w_{d+k-r+1}$ to $w_{d-r+1}$ by the left hand side of the diagram. In particular, this proves the claim for $r=1$. For general $r$, it follows by induction and the diagram
\begin{equation*}
\xymatrix{{H^*(\Sigma^k MT(d,r-1);\Z/2)}\ar[r]&{H^*(\Sigma^k MT(d,r);\Z/2)}\\
{H^*(MT(d+k,r-1);\Z/2)}\ar[r]\ar[u]^{f^*}&{H^*(MT(d+k,r);\Z/2).}\ar[u]^{f^*}}
\end{equation*}

In the oriented case, a similar argument shows that $f^*$ is a $(2(d-r+1)+k)$-equivalence on homology with $\Q$ and $\Z /p$ coefficients for $p>2$. So by the universal coefficient theorem, the same holds with integer coefficients. This implies that $f$ is a $(2(d-r+1)+k)$-equivalence.
\end{proof}
Theorem \ref{perfrembr} and the diagram \eqref{ftheta} implies:
\begin{cor}\label{peri}
If $a_r \mid k$, there is a commutative diagram
\begin{equation*}
\xymatrix{{\pi_q(\Sigma V_{d,r})}\ar[r]^-{\theta^r}\ar[d]^{f_{0*}}& {\pi_{q}(MTSO(d,r))}\ar[d]^{f_*} \\
{\pi_{q+k}(\Sigma V_{{d+k},r})}\ar[r]^-{\theta^r} & {\pi_{q+k}(MTSO(d+k,r))}.}
\end{equation*}
For all $q < 2(d-r+1)$, the vertical maps are isomorphisms.
\end{cor}

\subsection{Calculations in the oriented case}\label{calcSO}
We are now ready to calculate $\pi_*(MTSO(d,r))$ in low dimensions. The input for the Adams spectral sequence was $\Ext_{\A_p}^{s,t}(H^*(MTSO(d,r);\Z/p),\Z/p)$.
We first describe the action  of the mod $2$ Steenrod algebra $\mathcal{A}_2$ on $H^*(MTSO(d,r);\Z/2)$. Throughout this section, all cohomology groups have coefficients in $\Z/2$ unless explicitly indicated.

By Exercise 8-A in \cite{milnors}, the $\A_2$-action on $w_i \in H^i(BO(d))$ is given by the formula
\begin{equation}\label{binom}
Sq^k(w_i) = \sum_{j=0}^k \binom{i-k+j-1}{j} w_{i+j}w_{k-j}.
\end{equation}
Together with the Cartan formula, this yields formulas for the $\A_2$-action on all of $H^*(BSO(d))$. 
Consider the Thom isomorphism with Thom class $\bar{u}_d$
\begin{equation*}
\phi : H^i(BSO(d)) \to {H}^{i}(MT(d)).
\end{equation*}
It follows from Definition \ref{defsw} of the Stiefel--Whitney classes that
\begin{equation*}
Sq(\phi(x)) =\Sq(x \bar{u}_{d}) =\Sq(x) \Sq(\bar{u}_{d})= \phi (Sq(x) w(U_{d}^\perp) )
\end{equation*}
where $Sq=\sum \Sq^i$ denotes the total square, $w = \sum_i w_i$ is the total Stiefel--Whitney class, and
\begin{equation*}
w(U_{d}^\perp)=\varprojlim_n w(U_{d,n}^\perp ) \in \varprojlim_n \widetilde{H}^{*+n}(\Th(U_{d,n}^\perp)) \cong H^*(MT(d)).
\end{equation*}
This $w(U_d^\perp)$ must satisfy
\begin{equation*}
 w(U_{d})  w(U_{d}^\perp)=1,
\end{equation*}
so $w_i(U_{d}^\perp)$ can be computed from $w(U_{d})$ inductively.

In this way, we obtain formulas for the Steenrod action on $H^*(MTSO(d))$, and  these formulas hold in $H^*(MTSO(d,r))$ as well by naturality.

The  table below shows how $\mathcal{A}_2$ acts on $H^*(MTSO(d,r))$ in low dimensions. The middle column shows a $\Z/2$-basis in each dimension and the right column shows how $\mathcal{A}_2$ acts on these basis elements.
\begin{equation*}
\begin{tabular}{|c|c|r@\ c@\ l|}
\hline 
{$t$}&{$H^t(MTSO(d,r))$}&{}&&\\
\hline
{$d-r+1$}&{$w_{d-r+1}$}&&&\\
\hline
{$d-r+2$}&{$w_{d-r+2}$}&$Sq^1(w_{d-r+1})$&$ = $&$(d-r)w_{d-r+2}$\\
\hline
{$d-r+3$}&{$w_{d-r+3}$}&$Sq^2(w_{d-r+1})$&$ =$&$ \binom{d-r}{2}w_{d-r+3}$\\
{}&{$w_2w_{d-r+1}$}&$Sq^1(w_{d-r+2}) $&$=$&$ (d-r+1)w_{d-r+2}$\\
\hline
{$d-r+4$}&{$w_3w_{d-r+1}$}&$Sq^{1,2}(w_{d-r+1})$&$ = $&$(d-r)\binom{d-r}{2}w_{d-r+4}$\\
{}&{$w_2w_{d-r+2}$}&$Sq^{2,1}(w_{d-r+1})$&$ =$&$ (d-r+1)\binom{d-r+1}{2}w_{d-r+4}$\\
{}&{$w_{d-r+4}$}&$Sq^2(w_{d-r+2}) $&$=$&$ \binom{d-r+1}{2}w_{d-r+4}$\\
{}&{}&$Sq^1(w_{d-r+3}) $&$= $&$(d-r)w_{d-r+4}$\\
{}&{}&$Sq^1(w_2w_{d-r+1})$&$ =$&$ w_3w_{d-r+1} + (d-r)w_2w_{d-r+2}$\\
\hline
\end{tabular} 
\end{equation*}

The formulas depend on $d$ and $r$, so there is a number of special cases to consider.  
 To compute $\Ext_{\A_2}^{s,t}(H^*(MTSO(d,r)),\Z/2)$, we must construct a free resolution of $\A_2$-modules
\begin{equation*}
H^*(MTSO(d,r)) \xleftarrow{} F_0 \xleftarrow{} F_1 \xleftarrow{} F_2 \dotsm
\end{equation*}
Figure \ref{resolution} shows how to construct such a resolution in the special case $d \equiv 3 \bmod 4$ and $r=4$. First a basis for $H^{d-3}(MTSO(d,4))$ is chosen. Then all squares have been computed and, if necessary, a minimal number of classes in $H^{d-2}(MTSO(d,4))$ have been added to get a basis for $H^{d-2}(MTSO(d,4))$. Again all squares are computed and classes have been added to get a basis for $H^{d-1}(MTSO(d,4))$ and so on. These added basis elements form a generating set for $H^*(MTSO(d,4))$ as an $\mathcal{A}_2$-module. Let $F_0$ be the free module with a copy of $\mathcal{A}_2$ for each of these generators and $F_0 \to H^*(MTSO(d,4))$ the obvious surjection. Now the process is repeated with $H^*(MTSO(d,4))$ replaced by $\Ker(F_0 \to H^*(MTSO(d,4)))$ and $F_0$ replaced by $F_1$. Since the chosen resolution is minimal in the sense of \cite{hatcher2}, Lemma 2.8, 
\begin{equation*}
\Ext^{s,t}_{\mathcal{A}_2}(H^*(MTSO(d,4)),\Z/2) = \Hom^t(F_s,\Z/2).
\end{equation*}

In Figure \ref{resolution}, the grading of the modules is shown vertically and an element in $F_s$ is displayed next to its image in $F_{s-1}$.
When $s \geq 3$, $F_s$ is zero in dimensions less than or equal to $d+s$.
\begin{figure}
\begin{equation*}
\begin{tabular}{|c|r@\ c@\ l|c|c|c|}
\hline
$t$ & \multicolumn{3}{|c|}{ $H^t(MTSO(d,r))$} & $F_0$ & $F_1$ & $F_2$ \\
\hline
{$d-3$} &  &&{$w_{d-3}$}& {$x_1$} &  &  \\
\hline
$d-2$ &  $Sq^1 (w_{d-3})$&$ =$&$ w_{d-2}$ & $Sq^1(x_1)$ &  & \\
\hline
$d-1$ &  $Sq^2 (w_{d-3}) $&$=$&$ w_{d-1}$ & $Sq^2(x_1)$ &  &   \\
 &  && $w_2w_{d-3}$& $x_{2}$ & & \\
\hline
$d$ &  $Sq^{1,2} (w_{d-3})$&$ =$&$ w_d$ & $Sq^{1,2}(x_1)$ &  &  \\
 &  $Sq^{2,1} (w_{d-3})$&$ =$&$ 0$ & $Sq^{2,1} (x_1)$ & $\alpha_1$ & \\
 &  $Sq^1(w_2w_{d-3})$&$ =$&$ w_3 w_{d-3} + w_2 w_{d-2}$ & $Sq^1(x_{2})$ &  &  \\
 &   &&$w_3 w_{d-3}$& $x_{3}$ & & \\
\hline
$d+1$ &  $Sq^4 (w_{d-3})$&$ =$&$ w_3 w_{d-2} + w_2 w_{d-1}$& $Sq^4 (x_1)$ &  &  \\
 &  $Sq^{1,2,1}(w_{d-3})$&$ =$&$ 0$ & $Sq^{1,2,1}(x_1)$ & $Sq^{1} (\alpha_1)$ &  \\
 &  $Sq^{2}(w_2w_{d-3}) $&$ = $ &$w_2^2 w_{d-3} + w_3 w_{d-2}$& $Sq^{2}(x_{2})$ &  & \\
 &  &&$\ +w_2w_{d-1}$ &  &  & \\
 &  $Sq^1(w_3 w_{d-3})$&$=$&$ w_3 w_{d-2}$  & $Sq^1(x_3)$& & \\
 &    && $w_4 w_{d-3}$ & $x_{4}$ & &   \\
 \hline
$d+2$ & && &  & $Sq^{2} (\alpha_1)$ & $\beta_1$           \\
\hline
$d+3$ &  &&&  &  & $Sq^{1} (\beta_1)$           \\
\hline
\end{tabular}
\end{equation*}
\caption{Resolution of $H^*(MTSO(d,4))$ for $d \equiv 3\bmod 4$.}\label{resolution}
\end{figure}
The $E_2$-term of the Adams spectral sequence is shown in Figure \ref{r42}. Each dot represents a $\Z/2$ summand. A vertical line represents multiplication by $h_0$, while a sloped line indicates multiplication by $h_1$, see Section \ref{as}. This multiplication can be read off from the resolution. 

The resolutions in the remaining special cases have been omitted here since they would take up too much space. We shall give only the $E_2$-terms and explain how the differentials are determined and the extension problems are solved.


To do so, we need a couple of helpful facts. The free part of $\pi_*(MTSO(d,r))$ is known, since
\begin{equation}\label{Q}
H^*(X;\Q) \cong \pi_*^s(X) \otimes \Q
\end{equation}
for all spaces $X$, see e.g.\ \cite{hatcher2}, Chapter 1. 
The homotopy groups $\pi_q(V_{d,r})$ have been computed in \cite{paechter}.

Recall that $\pi_{q}(MTSO(d,1)) \cong \pi_q^s(S^d) \oplus \pi^s_q(\Sigma^d BSO(d))$. 
\begin{thm}\label{homotopyr1}
$\pi_{q}(MTSO(d,1)) $  is given by the following table for $q \leq 2d$:
\begin{equation*}
\begin{tabular}{ccccccc}
\hline
 {$q$}&{$d$}&{$d+1$}&{$d+2$}&{$d+3$}&{$d+4$}&{$d+5$} \\
\hline
 {$\pi_q^s(S^d)$}&{$\Z$}&{$\Z/2$}&{$\Z /2$}&{$\Z /24$} &{$0$}&{$0$} \\

 {$\pi_q^s(\Sigma^d BSO(d))$}&{$0$}&{$0$}&{$\Z /2$}&{$0$} &{$\Z \oplus \Z /2$} & {$\Z/2$}\\
\hline
\end{tabular}
\end{equation*}
\end{thm}
\begin{figure}
\begin{equation*}
\setlength{\unitlength}{1cm}
\begin{picture}(13,5)
\put(0.5,0.5){\vector(0,1){4}}
\put(0.5,0.5){\vector(1,0){5.5}}
\put(1.5,1){\circle*{0.1}}
\put(3.5,1){\circle*{0.1}}
\put(3.5,2){\circle*{0.1}}
\put(3.5,3){\circle*{0.1}}
\put(3.5,4.5){\makebox(0,0){$\vdots$}}
\put(3.5,4){\circle*{0.1}}
\put(4.5,2){\circle*{0.1}}
\put(4.5,4){$BSO(d)$}
\put(3.5,2){\line(0,1){2}}
\put(0,1){\makebox(0,0){$0$}}
\put(0,2){\makebox(0,0){$1$}}
\put(0,3){\makebox(0,0){$2$}}
\put(0,4.5){\makebox(0,0){$s$}}
\put(1.5,0){\makebox(0,0){$d+2$}}
\put(2.5,0){\makebox(0,0){$d+3$}}
\put(3.5,0){\makebox(0,0){$d+4$}}
\put(4.5,0){\makebox(0,0){$d+5$}}
\put(6,0){\makebox(0,0){$t-s$}}

\put(7.5,0.5){\vector(0,1){4}}
\put(7.5,0.5){\vector(1,0){5.5}}
\put(9.5,1){\circle*{0.1}}
\put(11.5,1){\circle*{0.1}}
\put(11.5,4){$C_\theta$}
\put(13,0){\makebox(0,0){$t-s$}}
\put(7,1){\makebox(0,0){$0$}}
\put(7,2){\makebox(0,0){$1$}}
\put(7,3){\makebox(0,0){$2$}}
\put(7,4.5){\makebox(0,0){$s$}}
\put(9.5,0){\makebox(0,0){$d-r+3$}}
\put(11.5,0){\makebox(0,0){$d-r+4$}}
\end{picture}
\end{equation*}
\caption{The Adams spectral sequence for $BSO(d)$ and $C_\theta$.} \label{r1}
 \end{figure}
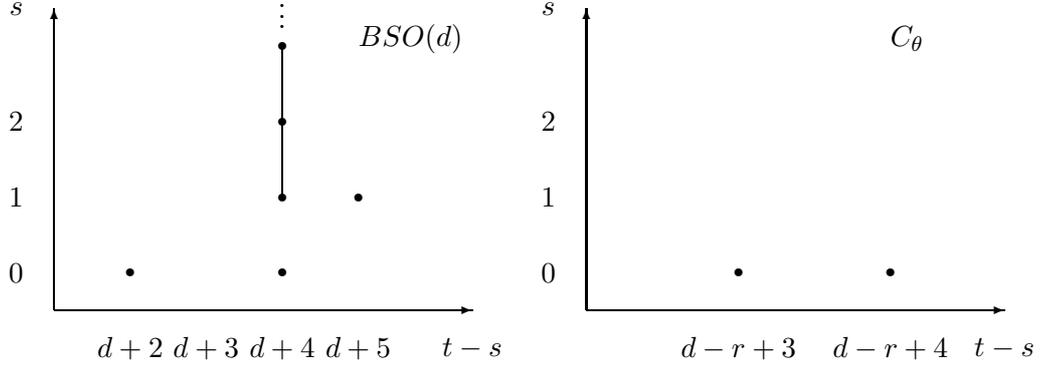
 
\begin{proof}
The first line is well-known.  Since $MTSO(d,1)\cong S^{\infty + d}\vee \Sigma^{\infty+d} BSO(d)$, the Adams spectral sequence splits as the sum of the sequence for $S^{\infty+d}$ and the sequence for $\Sigma^{\infty+d} BSO(d)$.  The $E_2$-term of the $BSO(d)$ part is shown in Figure~\ref{r1}. It follows from the formula $d_k(h_0x) = h_0d_k(x)$ that there can be no differentials in this part of the sequence. This yields the second line. 
\end{proof}

\begin{cor} \label{inj2}
The map
\begin{equation*}
\theta^2 : \pi_{d+1}(V_{d,2}) \to \pi_{d+2}(MTSO(d,2))
\end{equation*}
is injective for $d\geq 5$.
\end{cor}

\begin{proof}
Consider the diagram from Proposition \ref{egenskaberr}
\begin{equation*}\label{homotopy2}
\vcenter{\xymatrix@C=10pt{{\pi_{d+3}(MT(d,1))}\ar[r] &{\pi_{d+2}(MT(d-1,1))}\ar[r]&{\pi_{d+2}(MT(d,2))}\ar[r]&{\pi_{d+2}(MT(d,1))}\\
{\pi_{d+2}(V_{d,1})}\ar[r] \ar[u]^{\theta^1}&{\pi_{d+1}(V_{d-1,1})}\ar[u]^{\theta^1} \ar[r]&{\pi_{d+1}(V_{d,2})}\ar[u]^{\theta^2} \ar[r]&{\pi_{d+1}(V_{d,1}).} \ar[u]^{\theta^1} 
}}
\end{equation*}
The first two vertical maps are isomorphisms by Theorem \ref{homotopyr1} and the fourth is injective by Theorem \ref{theta-inj}, so the third is injective by the 5-lemma.
\end{proof}
 
\begin{lem} \label{homotopy-cofiber}
For $r>1$ and $q \leq 2(d-r)-1$ and $C_\theta $ the cofiber of \eqref{cofiber}, the homotopy groups $\pi_q(C_\theta)$ are given by:
\begin{equation*}
\begin{tabular}{cccc}
\hline
 {$q$}&{$ q\leq d-r+2$}&{$d-r+3$}&{$d-r+4$} \\
\hline
 {$\pi_q(C_\theta)$}&{$0$}&{$\Z /2$}&{$\Z /2$} \\
\hline
\end{tabular}
\end{equation*}
\end{lem}

\begin{proof}
The cohomology structure with $\Z/p$ coefficients is given in Corollary \ref{cofib2} for $p=2$, and otherwise it is the kernel of the map in Theorem \ref{cohom-cofiber2}. Only in the $p=2$ case, anything interesting happens in dimensions $q \leq d-r+4$. The relevant part of the spectral sequence is shown in Figure \ref{r1}. Obviously, there can be no differentials, and the claims follow.
\end{proof} 

\begin{cor}
For $d-r\geq 2$,
\begin{equation*}
\theta^r : \pi_{d-r+2}(V_{d,r}) \to \pi_{d-r+3}(MTSO(d,r))
\end{equation*}
is injective with cokernel $\Z/2$. In particular, it is not an isomorphism.
\end{cor}
This is contrary to the map $\tilde{\theta}^t_r$ defined by Atiyah and Dupont, which was an isomorphism in this dimension.
\begin{proof}
It follows from the long exact sequence
\begin{equation}\label{cofiber2}
\dotsm \to \pi_{q}^s(\Sigma V_{d,r}) \xrightarrow{f_{\theta*}} \pi_{q} (MT(d,r)) \to \pi_{q}(C_{\theta }) \to \pi_{q-1}^s(\Sigma V_{d,r}) \to \dotsm
\end{equation}
combined with Lemma \ref{homotopy-cofiber} and Corollary \ref{KRinj} that the cokernel is $\Z/2$.
\end{proof}

\begin{thm} \label{MTd2}
For $q < 2(d-1)$, $\pi_q(MTSO(d,2))$ is given by the table:
\begin{equation*}
\begin{tabular}{ccccc}
\hline
 {$q$}&{$d-1$}&{$d$}&{$d+1$}&$d+2$ \\
\hline
 {$d$ even}&{$\Z$}&{$\Z \oplus \Z/2$}&{$\Z /2 \oplus \Z /2 \oplus \Z /2$}&{$\Z/48 \oplus \Z/2$} \\

{$d$ odd}&{$\Z /2$}&{$\Z/2$} &{$\Z /4 \oplus \Z /2$}&{$\Z /4 \oplus \Z /2$}\\
\hline
\end{tabular}
\end{equation*}
\end{thm}

\begin{figure}[p]
\setlength{\unitlength}{1cm}
\begin{picture}(13,5)
\put(0.5,0.5){\vector(0,1){4.5}}
\put(0.5,0.5){\vector(1,0){5.5}}
\put(1.5,1){\circle*{0.1}}
\put(3.5,1){\circle*{0.1}}
\put(4.5,1){\circle*{0.1}}
\put(4.5,2){\circle*{0.1}}
\put(4.5,3){\circle*{0.1}}
\put(2.5,2){\circle*{0.1}}
\put(3.5,2){\circle*{0.1}}
\put(3.5,3){\circle*{0.1}}
\put(6,0){\makebox(0,0){$t-s$}}
\put(4.5,4.5){$d$ odd}
\put(3.5,2){\line(0,1){1}}
\put(1.5,1){\line(1,1){1}}
\put(2.5,2){\line(1,1){1}}
\put(4.5,1){\line(0,1){1}}
\put(3.5,2){\line(1,1){1}}
\put(0,1){\makebox(0,0){$0$}}
\put(0,2){\makebox(0,0){$1$}}
\put(0,3){\makebox(0,0){$2$}}
\put(0,4.5){\makebox(0,0){$s$}}
\put(1.5,0){\makebox(0,0){$d-1$}}
\put(2.5,0){\makebox(0,0){$d$}}
\put(3.5,0){\makebox(0,0){$d+1$}}
\put(4.5,0){\makebox(0,0){$d+2$}}

\put(7.5,0.5){\vector(0,1){4.5}}
\put(7.5,0.5){\vector(1,0){5.5}}
\put(8.5,1){\circle*{0.1}}
\put(8.5,2){\circle*{0.1}}
\put(8.5,3){\circle*{0.1}}
\put(8.5,4){\circle*{0.1}}
\put(8.5,4.5){\makebox(0,0){$\vdots$}}
\put(9.5,1){\circle*{0.1}}
\put(9.5,2){\circle*{0.1}}
\put(9.5,3){\circle*{0.1}}
\put(9.5,4){\circle*{0.1}}
\put(9.5,4.5){\makebox(0,0){$\vdots$}}
\put(9.7,2){\circle*{0.1}}
\put(10.5,1){\circle*{0.1}}
\put(11.5,1){\circle*{0.1}}
\put(11.5,2){\circle*{0.1}}
\put(11.5,3){\circle*{0.1}}
\put(11.5,4){\circle*{0.1}}
\put(11.7,3){\circle*{0.1}}
\put(10.5,2){\circle*{0.1}}
\put(10.5,3){\circle*{0.1}}
\put(13,0){\makebox(0,0){$t-s$}}
\put(7,1){\makebox(0,0){$0$}}
\put(7,2){\makebox(0,0){$1$}}
\put(7,3){\makebox(0,0){$2$}}
\put(7,4.5){\makebox(0,0){$s$}}
\put(8.5,0){\makebox(0,0){$d-1$}}
\put(9.5,0){\makebox(0,0){$d$}}
\put(10.5,0){\makebox(0,0){$d+1$}}
\put(11.5,0){\makebox(0,0){$d+2$}}
\put(11.5,4.5){$d$ even }
\put(8.5,1){\line(0,1){3}}
\put(9.5,1){\line(0,1){3}}
\put(8.5,1){\line(6,5){1.2}}
\put(9.7,2){\line(4,5){0.8}}
\put(10.5,3){\line(1,1){1}}
\put(11.5,1){\line(0,1){3}}
\put(9.5,1){\line(1,1){1}}
\put(10.5,2){\line(6,5){1.2}}
\end{picture}
\caption{The Adams spectral sequence, $r=2$}\label{figur}
\end{figure}
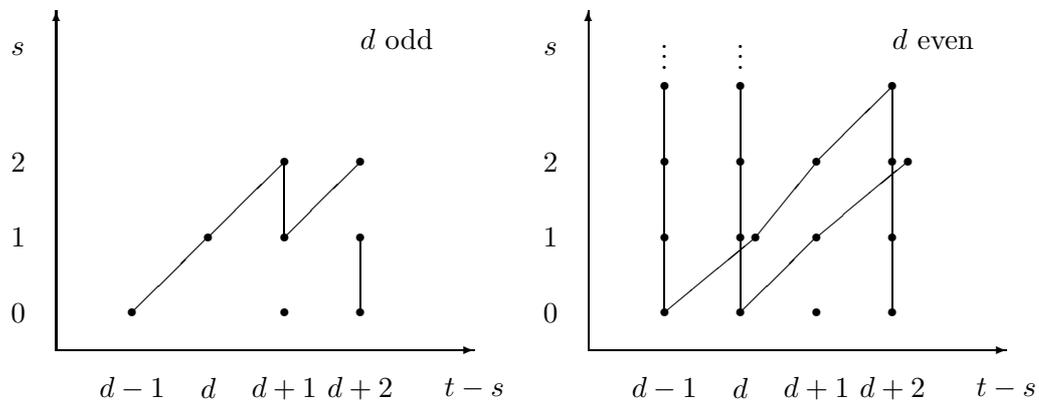

\begin{proof}
Again we consider the Adams spectral sequence for the 2-primary part. The $E_2$-term is shown in Figure \ref{figur}. We immediately see that the $(d-1)$th and the $d$th column must survive to $E_\infty$ in both cases. In the even case, this is true by \eqref{Q} and Theorem \ref{Fcoeff}.

For $d$ even, consider the exact sequence for the pair $(\Sigma^\infty \Sigma V_{d,2},MTSO(d,2))$. The homotopy groups of $V_{d,2}$ are known from \cite{paechter}, and the homotopy groups of the cofiber are calculated in Lemma \ref{homotopy-cofiber}. Inserting this in \eqref{cofiber2} yields:
\begin{equation*}
\begin{tabular}{c@\ c@\ c@\ c@\ c@\ c}
$\dotsm$&$\Z/24 \oplus \Z/2$ &$\to$ &$\pi_{d+2}(MTSO(d,2))$& $\to$ &$\Z/2$  \\
$\to$ &$\Z/2 \oplus \Z/2$ &$\to$ &$\pi_{d+1}(MTSO(d,2))$& $\to$& $\Z/2$\\ 
\end{tabular}
\end{equation*}
It follows from Theorem \ref{theta-inj} that the last map is surjective. By Corollary \ref{inj2}, the first map is injective and by Corollary \ref{KRinj}, the third map is zero. Thus the third and fourth column in the spectral sequence survive to $E_\infty$, otherwise the groups would be too small for the exact sequence. The odd case is similar.

Most of the necessary information about extensions is given by the vertical lines. In the even case, there may still be an extension problem in dimension $d+1$. Note, however, that under the map induced by $MTSO(d-1,1) \to MTSO(d,2)$ on spectral sequences, the generator of $E^{d+1,0}_\infty$ is hit by something representing a map of order two. But the map is induced by a filtration preserving map of homotopy groups. Thus, the generator cannot represent a map of order greater than two. 

When $d $ is odd and $p>2$, $H^*(MTSO(d,2);\Z/p)=0$ in the relevant dimensions, so there is no $p$-torsion. 

When $d$ is even, Theorem \ref{Fcoeff} yields
\begin{equation*}
H^{*}(MTSO(d,2);\Z/p) \cong H^{*}(BSO(d);\Z/p) \cdot \phi(\delta(e_{d-2}))  \oplus H^{*}(BSO(d);\Z/p) \cdot \phi(e_d) 
\end{equation*} 
in low dimensions. Since $\phi(\delta(e_{d-2}))$ maps to zero under the map 
\begin{equation*}
H^{d-1}(MTSO(d,2);\Z/p) \to H^{d-1}(MTSO(d);\Z/p),
\end{equation*}
so must all Steenrod powers of $\phi(\delta(e_{d-2}))$. That is, they all lie in the kernel, which is
\begin{equation*}
 H^*(BSO(d);\Z/p) \cdot \phi(\delta(e_{d-2})).
\end{equation*}
On the other hand, the map
\begin{equation*} 
H^*(MTSO(d,2);\Z/p) \to H^*(MTSO(d-1,1);\Z/p)
\end{equation*}
is injective on $ H^*(BSO(d);\Z/p) \cdot \pi(\delta(e_{d-2}))$. Since $\phi(\delta(e_{d-2}))$ maps to the generator of the $H^{d-1}(S^{d-1};\Z/p)$ summand and all Steenrod powers of this are zero, all powers of $\phi(\delta(e_{d-2}))$ must be zero in $H^{*}(MTSO(d,2),\Z/p)$ as well. Using this, the Adams spectral sequence immediately shows that the only $p$-torsion is a $\Z/3$ summand in dimension $d+2$.
\end{proof}

\begin{thm} \label{MTd3}
For $q < 2(d-2)$, $\pi_q(MTSO(d,3))$ is given by the following table:
\begin{equation*}
\begin{tabular}{ccccc}
\hline
 {$q$}&{$d-2$}&{$d-1$}&{$d$}&{$d+1$} \\
\hline
 {$d \equiv 0 \bmod 4$ }&{$\Z /2$}&{$\Z /2 $}&{$\Z \oplus \Z /4 \oplus \Z /2$} &{$\Z/2 \oplus \Z/2 \oplus \Z/2$}\\

{$d \equiv 1 \bmod 4$ }&{$\Z$}&{$\Z /4$}&{$\Z /2 \oplus \Z /2$} &{ $\Z/24 \oplus \Z/2$}\\

{$d \equiv 2 \bmod 4$ }&{$\Z/2$}&{$0$}&{$\Z \oplus \Z /2 \oplus \Z/2$}&{$\Z/2 \oplus \Z/2$} \\

{$d \equiv 3 \bmod 4$ }&{$\Z$}&{$\Z/2 \oplus \Z/2$}&{$\Z /2 \oplus \Z/2 \oplus \Z/2$} &{$\Z/48 \oplus \Z/4$}\\

\hline
\end{tabular}
\end{equation*}
\end{thm}

\begin{figure}[p]
\begin{equation*}
\setlength{\unitlength}{1cm}
\begin{picture}(13,5)
\put(0.5,0.5){\vector(0,1){4.5}}
\put(0.5,0.5){\vector(1,0){5.5}}
\put(1.5,1){\circle*{0.1}}
\put(2.5,2){\circle*{0.1}}
\put(3.5,1){\circle*{0.1}}
\put(3.5,2){\circle*{0.1}}
\put(3.5,3){\circle*{0.1}}
\put(3.7,1){\circle*{0.1}}
\put(3.7,2){\circle*{0.1}}
\put(3.7,3){\circle*{0.1}}
\put(3.7,4){\circle*{0.1}}
\put(3.7,4.5){\makebox(0,0){$\vdots$}}
\put(4.5,1){\circle*{0.1}}
\put(4.5,2){\circle*{0.1}}
\put(4.5,3){\circle*{0.1}}
\put(3.9,4.7){$d \equiv 0 \bmod 4$}
\put(3.5,2){\line(0,1){1}}
\put(3.7,1){\line(0,1){3}}
\put(1.5,1){\line(1,1){1}}
\put(2.5,2){\line(1,1){1}}
\put(3.7,1){\line(4,5){0,8}}
\put(3.5,2){\line(1,1){1}}
\put(0,1){\makebox(0,0){$0$}}
\put(0,2){\makebox(0,0){$1$}}
\put(0,3){\makebox(0,0){$2$}}
\put(0,4.5){\makebox(0,0){$s$}}
\put(1.5,0){\makebox(0,0){$d-2$}}
\put(2.5,0){\makebox(0,0){$d-1$}}
\put(3.5,0){\makebox(0,0){$d$}}
\put(4.5,0){\makebox(0,0){$d+1$}}
\put(6,0){\makebox(0,0){$t-s$}}

\put(7.5,0.5){\vector(0,1){4.5}}
\put(7.5,0.5){\vector(1,0){5.5}}
\put(8.5,1){\circle*{0.1}}
\put(8.5,2){\circle*{0.1}}
\put(8.5,3){\circle*{0.1}}
\put(8.5,4){\circle*{0.1}}
\put(8.5,4.5){\makebox(0,0){$\vdots$}}
\put(9.5,1){\circle*{0.1}}
\put(9.5,2){\circle*{0.1}}
\put(10.5,1){\circle*{0.1}}
\put(10.5,2){\circle*{0.1}}
\put(11.7,1){\circle*{0.1}}
\put(11.7,2){\circle*{0.1}}
\put(11.7,3){\circle*{0.1}}
\put(11.5,3){\circle*{0.1}}
\put(10.9,4.7){$d \equiv 1 \bmod 4$}
\put(9.5,1){\line(0,1){1}}
\put(8.5,1){\line(0,1){3}}
\put(11.7,1){\line(0,1){2}}
\put(8.5,1){\line(1,1){1}}
\put(9.5,1){\line(1,1){2}}
\put(13,0){\makebox(0,0){$t-s$}}
\put(7,1){\makebox(0,0){$0$}}
\put(7,2){\makebox(0,0){$1$}}
\put(7,3){\makebox(0,0){$2$}}
\put(7,4.5){\makebox(0,0){$s$}}
\put(8.5,0){\makebox(0,0){$d-2$}}
\put(9.5,0){\makebox(0,0){$d-1$}}
\put(10.5,0){\makebox(0,0){$d$}}
\put(11.5,0){\makebox(0,0){$d+1$}}
\end{picture}
\end{equation*}

\begin{equation*}
\setlength{\unitlength}{1cm}
\begin{picture}(13,5)
\put(0.5,0.5){\vector(0,1){4.5}}
\put(0.5,0.5){\vector(1,0){5.5}}
\put(1.5,1){\circle*{0.1}}
\put(3.5,1){\circle*{0.1}}
\put(3.7,2){\circle*{0.1}}
\put(3.7,3){\circle*{0.1}}
\put(3.7,4){\circle*{0.1}}
\put(3.7,4.5){\makebox(0,0){$\vdots$}}
\put(3.5,2){\circle*{0.1}}
\put(4.5,1){\circle*{0.1}}
\put(4.5,3){\circle*{0.1}}
\put(3.9,4.7){$d \equiv 2 \bmod 4$}
\put(3.5,2){\line(1,1){1}}
\put(3.7,2){\line(0,1){2}}
\put(0,1){\makebox(0,0){$0$}}
\put(0,2){\makebox(0,0){$1$}}
\put(0,3){\makebox(0,0){$2$}}
\put(0,4.5){\makebox(0,0){$s$}}
\put(1.5,0){\makebox(0,0){$d-2$}}
\put(2.5,0){\makebox(0,0){$d-1$}}
\put(3.5,0){\makebox(0,0){$d$}}
\put(4.5,0){\makebox(0,0){$d+1$}}
\put(6,0){\makebox(0,0){$t-s$}}

\put(7.5,0.5){\vector(0,1){4.5}}
\put(7.5,0.5){\vector(1,0){5.5}}
\put(8.5,1){\circle*{0.1}}
\put(8.5,2){\circle*{0.1}}
\put(8.5,3){\circle*{0.1}}
\put(8.5,4){\circle*{0.1}}
\put(8.5,4.5){\makebox(0,0){$\vdots$}}
\put(9.5,1){\circle*{0.1}}
\put(9.5,2){\circle*{0.1}}
\put(10.5,1){\circle*{0.1}}
\put(10.5,3){\circle*{0.1}}
\put(10.5,2){\circle*{0.1}}
\put(11.5,1){\circle*{0.1}}
\put(11.5,2){\circle*{0.1}}
\put(11.5,3){\circle*{0.1}}
\put(11.5,4){\circle*{0.1}}
\put(11.7,2){\circle*{0.1}}
\put(11.7,3){\circle*{0.1}}
\put(10.9,4.7){$d \equiv 3 \bmod 4$}
\put(8.5,1){\line(0,1){3}}
\put(11.7,2){\line(0,1){1}}
\put(11.5,1){\line(0,1){3}}
\put(8.5,1){\line(1,1){3}}
\put(9.5,1){\line(1,1){1}}
\put(10.5,2){\line(6,5){1.2}}
\put(13,0){\makebox(0,0){$t-s$}}
\put(7,1){\makebox(0,0){$0$}}
\put(7,2){\makebox(0,0){$1$}}
\put(7,3){\makebox(0,0){$2$}}
\put(7,4.5){\makebox(0,0){$s$}}
\put(8.5,0){\makebox(0,0){$d-2$}}
\put(9.5,0){\makebox(0,0){$d-1$}}
\put(10.5,0){\makebox(0,0){$d$}}
\put(11.5,0){\makebox(0,0){$d+1$}}
\end{picture}
\end{equation*}
\caption{The Adams spectral sequence, $r=3$.}\label{r32}
\end{figure}
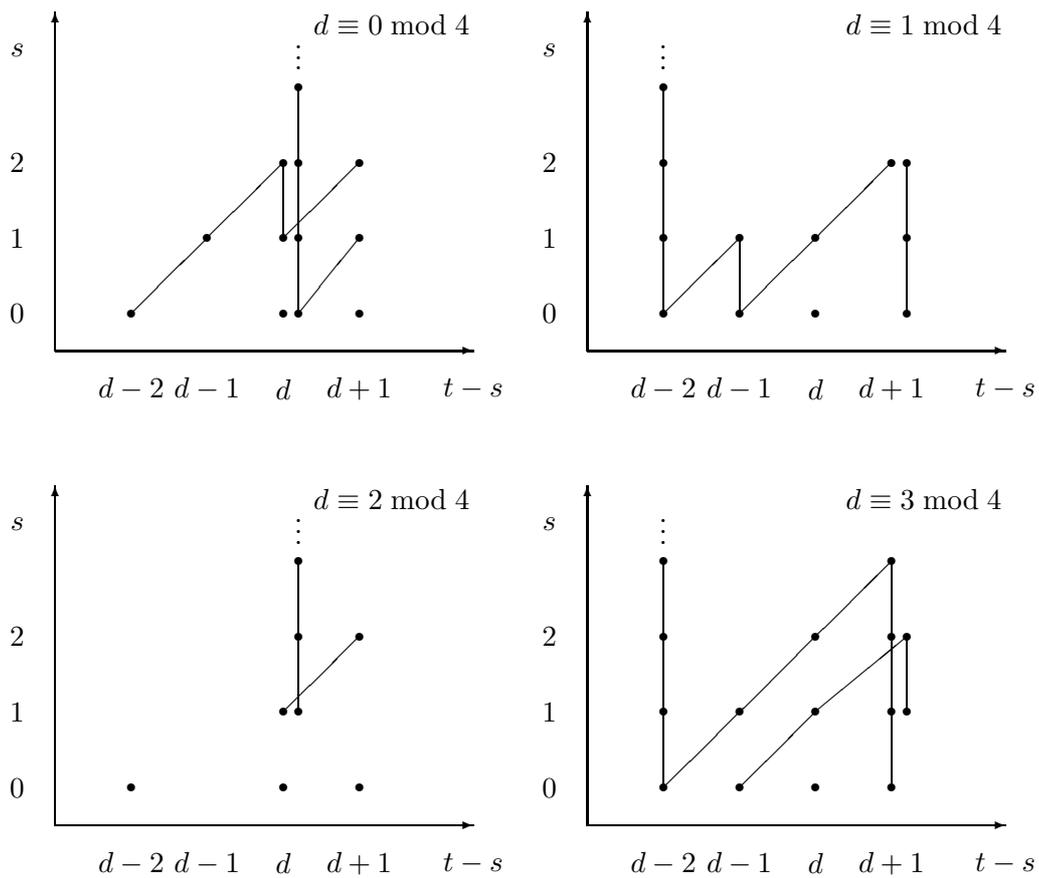

\begin{proof}
The spectral sequence for the $2$-primary part is shown in Figure~\ref{r32}.

For $d \equiv 1,2 \bmod 4$, we immediately see that the first three columns survive to $E_\infty$. For $d \equiv 1 \bmod 4$, we fill in the already computed groups in the exact sequence for the pair $(MTSO(d-1,2),MTSO(d,3))$. This yields
\begin{equation*}
\Z/2 \oplus \Z/2 \to \Z/16 \oplus \Z/2 \to \pi_{d+1}(MTSO(d,3)) \to \Z/2 ,  
\end{equation*} 
so $\pi_{d+1}(MTSO(d,3))$ must contain elements of order at least 8. Thus there can be no differentials hitting $h_0^2y$ where $y\in E_2^{d+1,0}$ is the generator. The other dot in $E^{d+1,2}_2$ is $h_1x$ for some non-zero $x \in E_2^{d,1}$. But under the map 
\begin{equation*}
g:MTSO(d,3)\to MTSO(d,2),
\end{equation*}
this $x$ is mapped to $g_*(x)$, which is non-zero in the spectral sequence for $MTSO(d,2)$, and $h_1g_*(x)$ survives to $E_\infty$. From the resolutions we see that $g_*(h_1x) = h_1g_*(x)$. Therefore, since $g_* d_2 = d_2g_*$, $h_1x$ is not hit by any differential. Also, $g_*(h_0^2y)=0$.  We deduce that $d_2: E_2^{d+2,0} \to E_2^{d+1,2}$ is zero.
 
For $d \equiv 2 \bmod 4$, the only unknown differential is $d_2: E_2^{d+2,0} \to E_2^{d+1,2}$. However, we may apply naturality of the Adams spectral sequence to the map $g : MTSO(d-2,1) \to MTSO(d,3)$. There is a $y \in E^{'d+2,0}_2$ in the spectral sequence $E'$ for $MTSO(d-2,1)$ that maps to the non-zero element $x \in E_2^{d+2,0}$. 
 We know $d_2(y)=0$ from the $r=1$ case, so 
\begin{equation*}
 d_2(x) =d_2(g_*(y)) = g_*(d_2(y))=0. 
\end{equation*}

In the case $d \equiv 0 \bmod 4$, there is a possible differential $d_2: E_2^{d+1,0} \to E_2^{d,2}$. However, we have a map $MTSO(d-1,2) \to MTSO(d,3)$. By comparing the spectral sequences, we see that this differential must be zero. The  differential $d_2: E_2^{d+2,0} \to E_2^{d+1,2}$ must be zero by a comparison with the spectral sequence for $MTSO(d-2,1)$. 

In the case $d \equiv 3 \bmod 4$, we see that the differential $d_2: E_2^{d+1,0} \to E_2^{d,2}$ must be zero by comparing with the spectral sequence for $MTSO(d-1,2)$. The extension problem in dimension $d$ is solved exactly as in the $r=2$ case.

The $p$-primary part for $p>2$ is determined precisely as in the $r=2$ case.

Finally there are extension problems in the even cases in dimension $d+1$. Insert the computed groups in the exact sequence for the cofibration 
\begin{equation*}
MTSO(d-2,1) \to MTSO(d,3) \to MTSO(d,2).
\end{equation*}
For $d \equiv 0 \mod 4$, this yields:
\begin{equation*}
\begin{tabular}{c@\ c@\ c@\ c@\ c@\ c}
$\dotsm$&$\Z/24$& $\to$ &$\pi_{d+1}(MTSO(d,3))$& $\to$ &$\Z/2 \oplus \Z/2 \oplus \Z/2$ \\
$\to$ & $\Z/2 \oplus \Z/2$  &$\to $  &$\Z \oplus \Z/4 \oplus \Z/2$ &$\to$ &$\Z \oplus \Z/2$\\
\end{tabular}
\end{equation*}
The third map must be zero, so $\pi_{d+1}(MTSO(d,3)) \to \Z/2 \oplus \Z/2 \oplus \Z/2$ is surjective, and thus an isomorphism.
The case $d \equiv 2 \bmod 4$ is similar.

The only remaining problem is when $d \equiv 3 \bmod 4$ and $q=d+1$. In this case, the result will follow from the cases $r=4,5$ in the next theorem.
\end{proof}

\begin{thm}\label{MTd4}
For  $q<2(d-3)$, $\pi_q(MTSO(d,4))$ is given by:
\begin{equation*}
\begin{tabular}{ccccc}
\hline
 {$q$}&{$d-3$}&{$d-2$}&{$d-1$}&{$d$} \\
\hline
 {$d \equiv 0 \bmod 4$}&{$\Z$}&{$\Z /2\oplus \Z /2$}&{$\Z /2 \oplus \Z /2 \oplus \Z /2 $}&{$\Z \oplus \Z/48 \oplus \Z/ 4$} \\

 {$d \equiv 1 \bmod 4$}&{$\Z /2$}&{$\Z /2$}&{$\Z /8 \oplus \Z /2$}& {$\Z /2 \oplus \Z /2$} \\

 {$d \equiv 2 \bmod 4$}&{$\Z$}&{$\Z/4$}&{$\Z/2$}&{$\Z \oplus \Z/24$} \\

 {$d \equiv 3 \bmod 4$}&{$\Z/2$}&{$0$}&{$\Z/2 \oplus \Z/2$}&{$\Z /2 \oplus \Z/2$}\\
\hline
\end{tabular}
\end{equation*}
For $d$ even and  $q<2(d-4)$, $\pi_q(MTSO(d,5))$  is given by:
\begin{equation*}
\begin{tabular}{cccccc}
\hline
 {$q$}&{$d-4$}&{$d-3$}&{$d-2$}&{$d-1$}&{$d$} \\
\hline
 {$d \equiv 0 \bmod 4$}&{$\Z/2$}&{$0$}&{$\Z /2 \oplus \Z /2 $}&{$ \Z/ 2 \oplus \Z/2$} &{$\Z \oplus \Z/8 \oplus \Z/2 \oplus \Z/2$}\\

 {$d \equiv 2 \bmod 4$}&{$\Z/2$}&{$\Z/2$}&{$\Z/2 \oplus \Z/8$}&{$\Z/2$} &{$\Z \oplus \Z/2$}\\
\hline
\end{tabular}
\end{equation*}
\end{thm}

\begin{figure}[p]
\begin{equation*}
\setlength{\unitlength}{1cm}
\begin{picture}(13,5)
\put(0.5,0.5){\vector(0,1){4.5}}
\put(0.5,0.5){\vector(1,0){5.5}}
\put(1.5,1){\circle*{0.1}}
\put(1.5,2){\circle*{0.1}}
\put(1.5,3){\circle*{0.1}}
\put(1.5,4){\circle*{0.1}}
\put(1.5,4.5){\makebox(0,0){$\vdots$}}
\put(2.5,1){\circle*{0.1}}
\put(2.5,2){\circle*{0.1}}
\put(3.5,1){\circle*{0.1}}
\put(3.5,2){\circle*{0.1}}
\put(3.5,3){\circle*{0.1}}
\put(4.3,1){\circle*{0.1}}
\put(4.3,2){\circle*{0.1}}
\put(4.3,3){\circle*{0.1}}
\put(4.3,4){\circle*{0.1}}
\put(4.5,1){\circle*{0.1}}
\put(4.5,2){\circle*{0.1}}
\put(4.5,3){\circle*{0.1}}
\put(4.5,4){\circle*{0.1}}
\put(4.5,4.5){\makebox(0,0){$\vdots$}}
\put(4.7,2){\circle*{0.1}}
\put(4.7,3){\circle*{0.1}}
\put(5.5,1){\circle*{0.1}}
\put(5.5,2){\circle*{0.1}}
\put(5.5,3){\circle*{0.1}}
\put(5.5,4){\circle*{0.1}}
\put(5.5,4.5){\makebox(0,0){$\vdots$}}
\put(5.7,1){\circle*{0.1}}
\put(5.7,2){\circle*{0.1}}
\put(5.7,3){\circle*{0.1}}
\put(3.9,4.7){$d \equiv 0 \bmod 4$}
\put(4.3,1){\line(0,1){3}}
\put(4.7,2){\line(0,1){1}}
\put(4.5,1){\line(0,1){3}}
\put(1.5,1){\line(0,1){3}}
\put(5.5,2){\line(0,1){2}}
\put(1.5,1){\line(1,1){2}}
\put(3.5,3){\line(4,5){0.8}}
\put(2.5,1){\line(1,1){1}}
\put(3.5,2){\line(6,5){1.2}}
\put(4.5,1){\line(6,5){1.2}}
\put(4.7,2){\line(1,1){1}}
\put(0,1){\makebox(0,0){$0$}}
\put(0,2){\makebox(0,0){$1$}}
\put(0,3){\makebox(0,0){$2$}}
\put(0,4.5){\makebox(0,0){$s$}}
\put(1.5,0){\makebox(0,0){$d-3$}}
\put(2.5,0){\makebox(0,0){$d-2$}}
\put(3.5,0){\makebox(0,0){$d-1$}}
\put(4.5,0){\makebox(0,0){$d$}}
\put(6,0){\makebox(0,0){$t-s$}}

\put(7.5,0.5){\vector(0,1){4.5}}
\put(7.5,0.5){\vector(1,0){5.5}}
\put(8.5,1){\circle*{0.1}}
\put(9.5,2){\circle*{0.1}}
\put(10.5,3){\circle*{0.1}}
\put(10.5,2){\circle*{0.1}}
\put(10.5,1){\circle*{0.1}}
\put(11.5,2){\circle*{0.1}}
\put(11.5,1){\circle*{0.1}}
\put(10.7,1){\circle*{0.1}}
\put(10.9,4.7){$d \equiv 1 \bmod 4$}
\put(10.5,1){\line(0,1){2}}
\put(8.5,1){\line(1,1){2}}
\put(10.5,1){\line(1,1){1}}
\put(13,0){\makebox(0,0){$t-s$}}
\put(7,1){\makebox(0,0){$0$}}
\put(7,2){\makebox(0,0){$1$}}
\put(7,3){\makebox(0,0){$2$}}
\put(7,4.5){\makebox(0,0){$s$}}
\put(8.5,0){\makebox(0,0){$d-3$}}
\put(9.5,0){\makebox(0,0){$d-2$}}
\put(10.5,0){\makebox(0,0){$d-1$}}
\put(11.5,0){\makebox(0,0){$d$}}
\end{picture}
\end{equation*}

\begin{equation*}
\setlength{\unitlength}{1cm}
\begin{picture}(13,5)
\put(0.5,0.5){\vector(0,1){4.5}}
\put(0.5,0.5){\vector(1,0){5.5}}
\put(1.5,1){\circle*{0.1}}
\put(1.5,2){\circle*{0.1}}
\put(1.5,3){\circle*{0.1}}
\put(1.5,4){\circle*{0.1}}
\put(1.5,4.5){\makebox(0,0){$\vdots$}}
\put(2.5,1){\circle*{0.1}}
\put(2.5,2){\circle*{0.1}}
\put(3.5,1){\circle*{0.1}}
\put(4.7,1){\circle*{0.1}}
\put(4.7,2){\circle*{0.1}}
\put(4.7,3){\circle*{0.1}}
\put(4.5,2){\circle*{0.1}}
\put(4.5,3){\circle*{0.1}}
\put(4.5,4){\circle*{0.1}}
\put(4.5,4.5){\makebox(0,0){$\vdots$}}
\put(5.5,1){\circle*{0.1}}
\put(5.5,2){\circle*{0.1}}
\put(5.5,3){\circle*{0.1}}
\put(5.5,4){\circle*{0.1}}
\put(5.5,4.5){\makebox(0,0){$\vdots$}}
\put(5.7,1){\circle*{0.1}}
\put(3.9,4.7){$d \equiv 2 \bmod 4$}
\put(4.7,1){\line(0,1){2}}
\put(4.5,2){\line(0,1){2}}
\put(1.5,1){\line(0,1){3}}
\put(5.5,1){\line(0,1){3}}
\put(1.5,1){\line(1,1){1}}
\put(2.5,1){\line(0,1){1}}
\put(0,1){\makebox(0,0){$0$}}
\put(0,2){\makebox(0,0){$1$}}
\put(0,3){\makebox(0,0){$2$}}
\put(0,4.5){\makebox(0,0){$s$}}
\put(1.5,0){\makebox(0,0){$d-3$}}
\put(2.5,0){\makebox(0,0){$d-2$}}
\put(3.5,0){\makebox(0,0){$d-1$}}
\put(4.5,0){\makebox(0,0){$d$}}
\put(6,0){\makebox(0,0){$t-s$}}

\put(7.5,0.5){\vector(0,1){5}}
\put(7.5,0.5){\vector(1,0){6}}
\put(8.5,1){\circle*{0.1}}
\put(10.5,2){\circle*{0.1}}
\put(10.5,1){\circle*{0.1}}
\put(11.5,3){\circle*{0.1}}
\put(11.5,1){\circle*{0.1}}
\put(12.5,3){\circle*{0.1}}
\put(12.5,1){\circle*{0.1}}
\put(12.5,2){\circle*{0.1}}
\put(12.7,2){\circle*{0.1}}
\put(12.7,3){\circle*{0.1}}
\put(12.7,4){\circle*{0.1}}
\put(12.3,2){\circle*{0.1}}
\put(12.5,4){\circle*{0.1}}
\put(12.7,4.5){\makebox(0,0){$\vdots$}}
\put(10.9,4.7){$d \equiv 3 \bmod 4$}
\put(12.7,2){\line(0,1){2}}
\put(10.5,2){\line(1,1){2}}
\put(11.5,1){\line(4,5){0.8}}
\put(12.5,2){\line(0,1){2}}
\put(13,0){\makebox(0,0){$t-s$}}
\put(7,1){\makebox(0,0){$0$}}
\put(7,2){\makebox(0,0){$1$}}
\put(7,3){\makebox(0,0){$2$}}
\put(7,4.5){\makebox(0,0){$s$}}
\put(8.5,0){\makebox(0,0){$d-3$}}
\put(9.5,0){\makebox(0,0){$d-2$}}
\put(10.5,0){\makebox(0,0){$d-1$}}
\put(11.5,0){\makebox(0,0){$d$}}
\end{picture}
\end{equation*}
\caption{The Adams spectral sequence, $r=4$.} \label{r42}
 \end{figure}
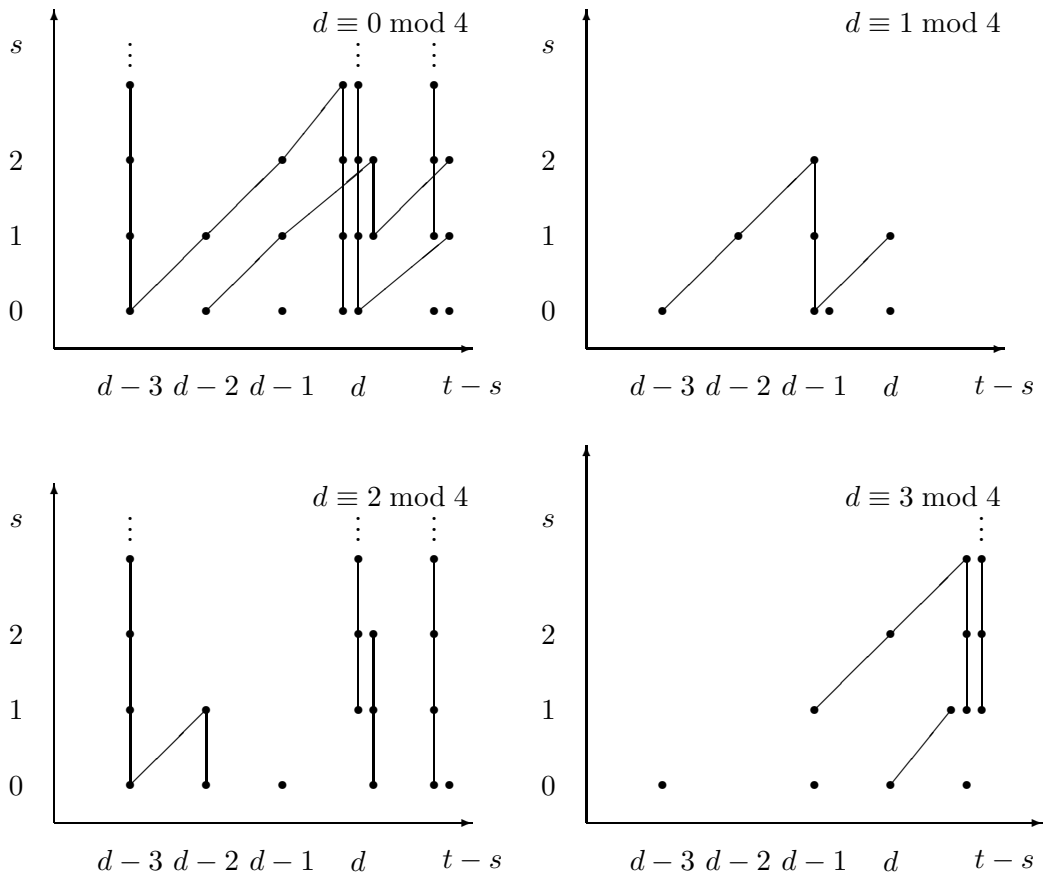

 \begin{figure}[p]
\begin{equation*}
\setlength{\unitlength}{1cm}
\begin{picture}(13,5)
\put(0.5,0.5){\vector(0,1){4.5}}
\put(0.5,0.5){\vector(1,0){5.5}}
\put(1.5,1){\circle*{0.1}}
\put(3.5,1){\circle*{0.1}}
\put(3.5,2){\circle*{0.1}}
\put(4.5,1){\circle*{0.1}}
\put(4.5,3){\circle*{0.1}}
\put(5.5,1){\circle*{0.1}}
\put(5.3,2){\circle*{0.1}}
\put(5.7,1){\circle*{0.1}}
\put(5.7,4.5){\makebox(0,0){$\vdots$}}
\put(5.7,2){\circle*{0.1}}
\put(5.7,3){\circle*{0.1}}
\put(5.7,4){\circle*{0.1}}
\put(5.5,2){\circle*{0.1}}
\put(5.5,3){\circle*{0.1}}
\put(5.5,4){\circle*{0.1}}
\put(3.9,4.7){$d \equiv 0 \bmod 4$}
\put(5.7,1){\line(0,1){3}}
\put(5.5,2){\line(0,1){2}}
\put(3.5,2){\line(1,1){2}}
\put(4.5,1){\line(4,5){0.8}}
\put(0,1){\makebox(0,0){$0$}}
\put(0,2){\makebox(0,0){$1$}}
\put(0,3){\makebox(0,0){$2$}}
\put(0,4.5){\makebox(0,0){$s$}}
\put(1.5,0){\makebox(0,0){$d-4$}}
\put(2.5,0){\makebox(0,0){$d-3$}}
\put(3.5,0){\makebox(0,0){$d-2$}}
\put(4.5,0){\makebox(0,0){$d-1$}}
\put(6,0){\makebox(0,0){$t-s$}}

\put(7.5,0.5){\vector(0,1){5}}
\put(7.5,0.5){\vector(1,0){6}}
\put(8.5,1){\circle*{0.1}}
\put(9.5,2){\circle*{0.1}}
\put(10.5,1){\circle*{0.1}}
\put(10.5,2){\circle*{0.1}}
\put(10.5,3){\circle*{0.1}}
\put(10.7,1){\circle*{0.1}}
\put(11.5,1){\circle*{0.1}}
\put(12.5,1){\circle*{0.1}}
\put(12.5,2){\circle*{0.1}}
\put(12.5,3){\circle*{0.1}}
\put(12.5,4){\circle*{0.1}}
\put(12.5,4.5){\makebox(0,0){$\vdots$}}
\put(10.9,4.7){$d \equiv 2 \bmod 4$}
\put(12.5,2){\line(0,1){2}}
\put(10.5,1){\line(0,1){2}}
\put(8.5,1){\line(1,1){2}}
\put(12.5,2){\line(0,1){2}}
\put(13,0){\makebox(0,0){$t-s$}}
\put(7,1){\makebox(0,0){$0$}}
\put(7,2){\makebox(0,0){$1$}}
\put(7,3){\makebox(0,0){$2$}}
\put(7,4.5){\makebox(0,0){$s$}}
\put(8.5,0){\makebox(0,0){$d-4$}}
\put(9.5,0){\makebox(0,0){$d-3$}}
\put(10.5,0){\makebox(0,0){$d-2$}}
\put(11.5,0){\makebox(0,0){$d-1$}}
\end{picture}
\end{equation*}
\caption{The Adams spectral sequence, $r=5$.} \label{r5}
 \end{figure}
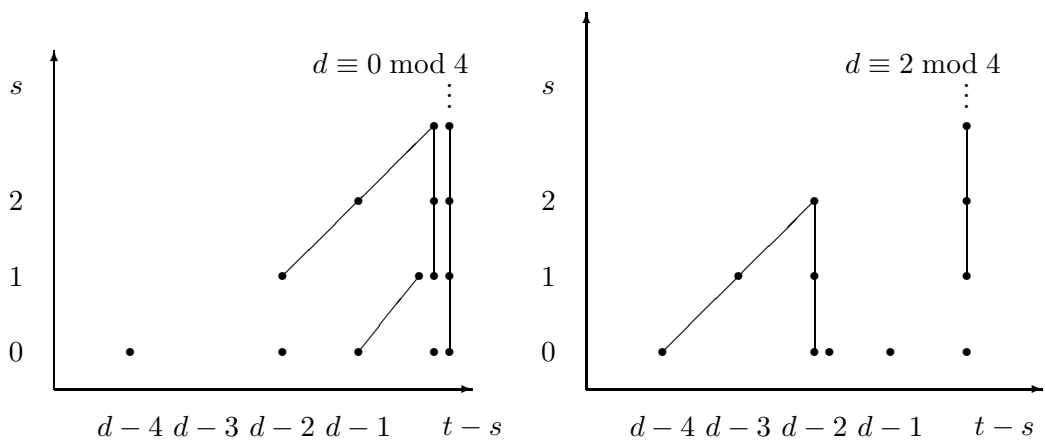

\begin{proof}
The calculation of the $p$-primary part for $p>2$ is again similar to the case $r=2$.
The spectral sequences for $p=2$ are shown in Figure \ref{r42} and \ref{r5}. For $r=4$, the first two columns in all four diagrams must survive to $E_\infty$.

For $d \equiv 1 \bmod 4$, the only differential is $d_2 : E_2^{d,0} \to E_2^{d-1,2}$, but this is zero by a comparison with the spectral sequence for $MTSO(d-1,3)$.

When $d \equiv 3 \bmod 4$, there is a differential $d_2: E_2^{d+1,0} \to E_2^{d,2}$. Comparing with the spectral sequence for $MTSO(d-3,1)$ shows that it must be zero. The extension problem in dimension $d$ is solved using the case $MTSO(d-1,3)$.

When $d \equiv 2 \bmod 4$, there may be a non-zero $d_2: E_2^{d+1,0} \to E_2^{d,2}$. But one of the generators of $E_2^{d+1,0}$ is in the image of the spectral sequence for  $MTSO(d-3,1)$, while the other generator is in the image of the spectral sequence for $MTSO(d-1,3)$ so by earlier computations this must be zero. 

Finally, for $d \equiv 0 \bmod 4$, the differential $d_2 : E_2^{d,0} \to E_2^{d-1,2}$ must be zero. One of the generators of $E_2^{d,0}$ is in the image of the spectral sequence for $MTSO(d-1,3)$, and the other one is in the image of the spectral sequence for $\Sigma^\infty \Sigma V_{d,4}$ under the map induced by $f_\theta$. Using our knowledge of the homotopy groups of $V_{d,4}$ and thus the spectral sequence, the differentials must be zero. The extension problem in dimension $d-1$ is solved as in the case $r=2$.

The $d$th column is a bit more complicated. One of the generators in $E_2^{d+1,1}$ is in the image of the spectral sequence for $MTSO(d-3,1)$. The other one is $h_1  x$ for some $x \in E_2^{d,0}$. Thus $d_2(h_1x) = h_1 d_2(x) + d_2(h_1)x = 0$.
In $E_2^{d+1,0}$, one of the generators is in the image of the spectral sequence for $MTSO(d-3,1)$, while the other one is in the image of the spectral sequence for $MTSO(d,5)$. Thus it is enough that the corresponding differential is zero for $MTSO(d,5)$. 

The results for $MTSO(d,5)$ for $d$ even again follow by comparing the spectral sequences with those for lower $r$. The only problem is when $d \equiv 0 \bmod 4$. By Corollary \ref{KRinj}, 
\begin{equation*}
\theta^5 : \pi_{d-1}(V_{d,5}) \to \pi_d(MTSO(d,5))
\end{equation*}
is injective when $d$ is divisible by 8. Since $\pi_{d-1}(V_{d,5})  \cong \Z \oplus \Z/8$, $\pi_d(MTSO(d,5))$ must contain torsion of order eight. Thus  $d_k:E_k^{d+1,0}\to E_k^{d,k}$ must be zero. This implies that the remaining differential for $MTSO(d,4)$ must be zero when $8\mid d$. By Corollary \ref{peri}, this holds for all $d \equiv 0 \bmod 4$. 

The next theorem shows that $\theta ^5$ is also injective when $d \equiv 4 \bmod 8$, proving the claim also for $\pi_d(MTSO(d,5))$ when $8 \nmid d$ by a similar argument. 
\end{proof}
%
The above computation allows us to improve Theorem \ref{theta-inj} further.

\begin{thm}\label{injtr}
For $2r \leq d+1$ 
\begin{equation*}
\theta^r : \pi_{d-1}(V_{d,r}) \to \pi_d(MTSO(d,r))
\end{equation*}
is injective for $d$ even and $r \leq 6$. For $d $ odd and $r=4$, it is not injective.
\end{thm}

\begin{proof}
For $d \equiv 2 \bmod 4$, the exact sequence for the pair $(MTSO(d,4),\Sigma^\infty \Sigma V_{d,4})$ in dimension $d$ becomes
\begin{equation*}
\to \Z \oplus \Z/12  \xrightarrow{\theta^4} \Z \oplus \Z/24 \to \Z/2 \to 0.
\end{equation*} 
Since this is exact, $\theta ^4$ must be injective.

For $d \equiv 0 \bmod 4$, the exact sequence for the pair $(MTSO(d,4),\Sigma^\infty \Sigma V_{d,4})$ in dimension $d$ becomes
\begin{equation*}
\to \Z \oplus \Z/24 \oplus \Z/4  \xrightarrow{\theta^4} \Z \oplus \Z/48 \oplus \Z/4 \to \Z/2 \to 0.
\end{equation*} 
Again this is exact, so $\theta ^4$ must be injective.

For $r =5,6$, the claim follows as in the proof of Theorem \ref{theta-inj} by a diagram similar to \eqref{spindia}.

For $d$ odd, the long exact sequence for the pair $(MTSO(d,4),\Sigma^\infty \Sigma V_{d,4})$ in dimension $d$ yields
\begin{equation*}
\to \Z/2 \oplus \Z/2 \xrightarrow{\theta^4 } \Z/2 \oplus \Z/2 \to \Z/2  \to
\end{equation*}
where the last map is surjective. We see that $\theta ^4$ cannot be injective. 
\end{proof}

\begin{rem}\label{remark}
Consider the case $r=4$ and $d \equiv 3 \bmod 8$. Then $S^d$ is Reinhart cobordant to $\emptyset$, so by Proposition \ref{betared}
\begin{equation*}
\beta^4(S^d)=\theta^4(\ind(s))=0,
\end{equation*}
even though  $S^d$ does not allow four independent vector fields according to Theorem~\ref{vfsp}. Hence $\theta^4$ loses essential information in this case. 
\end{rem}

\subsection{The unoriented case}\label{calcO}
By similar calculations of the Steenrod action on $H^*(MTO(d,r);\Z/2)$, we obtain the $E_2$-terms of Adams spectral sequences shown in Figure \ref{r1unor} and \ref{Assunor}. This immediately yields the first three homotopy groups. For the fourth, the only problem is to determine the differential $d_2 : E_2^{0,d-r+4} \to E_2^{2,d-r+5}$. We have not found a way to do this. 

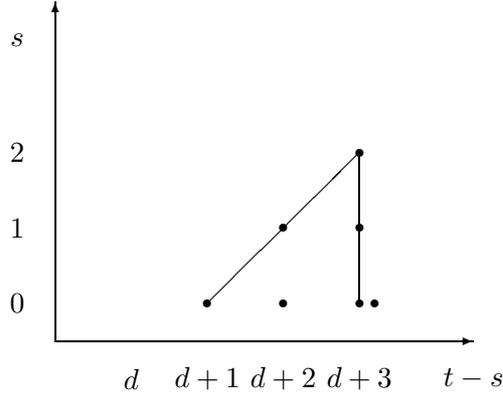
\begin{figure}
\begin{equation*}
\setlength{\unitlength}{1cm}
\begin{picture}(6,5)
\put(0.5,0.5){\vector(0,1){4.5}}
\put(0.5,0.5){\vector(1,0){5.5}}

\put(2.5,1){\circle*{0.1}}

\put(3.5,1){\circle*{0.1}}
\put(3.5,2){\circle*{0.1}}

\put(4.5,1){\circle*{0.1}}
\put(4.7,1){\circle*{0.1}}
\put(4.5,2){\circle*{0.1}}
\put(4.5,3){\circle*{0.1}}

\put(4.5,1){\line(0,1){2}}
\put(2.5,1){\line(1,1){2}}
\put(0,1){\makebox(0,0){$0$}}
\put(0,2){\makebox(0,0){$1$}}
\put(0,3){\makebox(0,0){$2$}}
\put(0,4.5){\makebox(0,0){$s$}}
\put(1.5,0){\makebox(0,0){$d$}}
\put(2.5,0){\makebox(0,0){$d+1$}}
\put(3.5,0){\makebox(0,0){$d+2$}}
\put(4.5,0){\makebox(0,0){$d+3$}}
\put(6,0){\makebox(0,0){$t-s$}}
\end{picture}
\end{equation*}
\caption{The Adams spectral sequence for $\Sigma^{d+\infty} BO$.} \label{r1unor}
\end{figure}

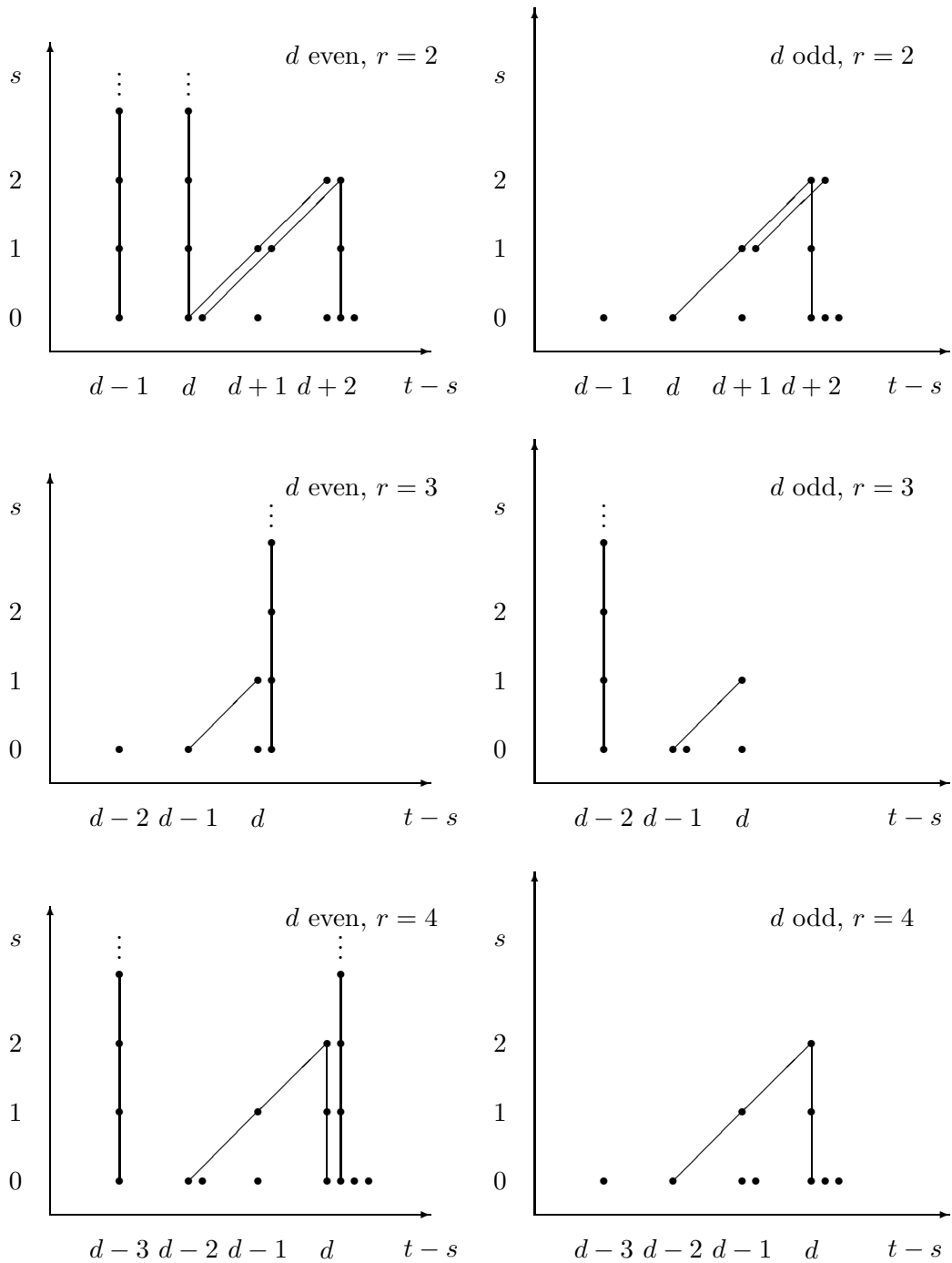
\begin{figure}
\begin{equation*}
\setlength{\unitlength}{1cm}
\begin{picture}(13,5)
\put(0.5,0.5){\vector(0,1){4.5}}
\put(0.5,0.5){\vector(1,0){5.5}}

\put(1.5,1){\circle*{0.1}}
\put(1.5,2){\circle*{0.1}}
\put(1.5,3){\circle*{0.1}}
\put(1.5,4){\circle*{0.1}}
\put(1.5,4.5){\makebox(0,0){$\vdots$}}

\put(2.5,1){\circle*{0.1}}
\put(2.5,2){\circle*{0.1}}
\put(2.5,3){\circle*{0.1}}
\put(2.5,4){\circle*{0.1}}
\put(2.5,4.5){\makebox(0,0){$\vdots$}}
\put(2.7,1){\circle*{0.1}}

\put(3.5,1){\circle*{0.1}}
\put(3.5,2){\circle*{0.1}}
\put(3.7,2){\circle*{0.1}}

\put(4.5,1){\circle*{0.1}}
\put(4.7,1){\circle*{0.1}}
\put(4.9,1){\circle*{0.1}}
\put(4.7,2){\circle*{0.1}}
\put(4.5,3){\circle*{0.1}}
\put(4.7,3){\circle*{0.1}}

\put(3.9,4.7){$d$ even, $r=2$}
\put(4.7,1){\line(0,1){2}}
\put(1.5,1){\line(0,1){3}}
\put(2.5,1){\line(0,1){3}}
\put(2.5,1){\line(1,1){2}}
\put(2.7,1){\line(1,1){2}}
\put(2.5,1){\line(0,1){1}}
\put(0,1){\makebox(0,0){$0$}}
\put(0,2){\makebox(0,0){$1$}}
\put(0,3){\makebox(0,0){$2$}}
\put(0,4.5){\makebox(0,0){$s$}}
\put(1.5,0){\makebox(0,0){$d-1$}}
\put(2.5,0){\makebox(0,0){$d$}}
\put(3.5,0){\makebox(0,0){$d+1$}}
\put(4.5,0){\makebox(0,0){$d+2$}}
\put(6,0){\makebox(0,0){$t-s$}}

\put(7.5,0.5){\vector(0,1){5}}
\put(7.5,0.5){\vector(1,0){6}}
\put(8.5,1){\circle*{0.1}}
\put(9.5,1){\circle*{0.1}}
\put(10.5,1){\circle*{0.1}}
\put(10.7,2){\circle*{0.1}}
\put(10.5,2){\circle*{0.1}}

\put(11.5,1){\circle*{0.1}}
\put(11.5,2){\circle*{0.1}}
\put(11.5,3){\circle*{0.1}}
\put(11.7,1){\circle*{0.1}}
\put(11.7,3){\circle*{0.1}}
\put(11.9,1){\circle*{0.1}}

\put(10.9,4.7){$d $ odd,  $r=2$}
\put(9.5,1){\line(1,1){2}}
\put(10.7,2){\line(1,1){1}}
\put(11.5,1){\line(0,1){2}}
\put(13,0){\makebox(0,0){$t-s$}}
\put(7,1){\makebox(0,0){$0$}}
\put(7,2){\makebox(0,0){$1$}}
\put(7,3){\makebox(0,0){$2$}}
\put(7,4.5){\makebox(0,0){$s$}}
\put(8.5,0){\makebox(0,0){$d-1$}}
\put(9.5,0){\makebox(0,0){$d$}}
\put(10.5,0){\makebox(0,0){$d+1$}}
\put(11.5,0){\makebox(0,0){$d+2$}}
\end{picture}
\end{equation*}

\begin{equation*}
\setlength{\unitlength}{1cm}
\begin{picture}(13,5)
\put(0.5,0.5){\vector(0,1){4.5}}
\put(0.5,0.5){\vector(1,0){5.5}}

\put(1.5,1){\circle*{0.1}}

\put(2.5,1){\circle*{0.1}}

\put(3.5,1){\circle*{0.1}}
\put(3.5,2){\circle*{0.1}}
\put(3.7,1){\circle*{0.1}}
\put(3.7,2){\circle*{0.1}}
\put(3.7,3){\circle*{0.1}}
\put(3.7,4){\circle*{0.1}}
\put(3.7,4.5){\makebox(0,0){$\vdots$}}

\put(3.9,4.7){$d$ even, $r=3$}

\put(3.7,1){\line(0,1){3}}
\put(2.5,1){\line(1,1){1}}
\put(0,1){\makebox(0,0){$0$}}
\put(0,2){\makebox(0,0){$1$}}
\put(0,3){\makebox(0,0){$2$}}
\put(0,4.5){\makebox(0,0){$s$}}
\put(1.5,0){\makebox(0,0){$d-2$}}
\put(2.5,0){\makebox(0,0){$d-1$}}
\put(3.5,0){\makebox(0,0){$d$}}

\put(6,0){\makebox(0,0){$t-s$}}

\put(7.5,0.5){\vector(0,1){5}}
\put(7.5,0.5){\vector(1,0){6}}
\put(8.5,1){\circle*{0.1}}
\put(8.5,2){\circle*{0.1}}
\put(8.5,3){\circle*{0.1}}
\put(8.5,4){\circle*{0.1}}
\put(8.5,4.5){\makebox(0,0){$\vdots$}}

\put(9.5,1){\circle*{0.1}}
\put(9.7,1){\circle*{0.1}}

\put(10.5,1){\circle*{0.1}}
\put(10.5,2){\circle*{0.1}}

\put(10.9,4.7){$d $ odd, $r=3$}
\put(9.5,1){\line(1,1){1}}
\put(8.5,1){\line(0,1){3}}
\put(13,0){\makebox(0,0){$t-s$}}
\put(7,1){\makebox(0,0){$0$}}
\put(7,2){\makebox(0,0){$1$}}
\put(7,3){\makebox(0,0){$2$}}
\put(7,4.5){\makebox(0,0){$s$}}
\put(8.5,0){\makebox(0,0){$d-2$}}
\put(9.5,0){\makebox(0,0){$d-1$}}
\put(10.5,0){\makebox(0,0){$d$}}

\end{picture}
\end{equation*}

\begin{equation*}
\setlength{\unitlength}{1cm}
\begin{picture}(13,5)
\put(0.5,0.5){\vector(0,1){4.5}}
\put(0.5,0.5){\vector(1,0){5.5}}

\put(1.5,1){\circle*{0.1}}
\put(1.5,2){\circle*{0.1}}
\put(1.5,3){\circle*{0.1}}
\put(1.5,4){\circle*{0.1}}
\put(1.5,4.5){\makebox(0,0){$\vdots$}}

\put(2.5,1){\circle*{0.1}}
\put(2.7,1){\circle*{0.1}}

\put(3.5,1){\circle*{0.1}}
\put(3.5,2){\circle*{0.1}}

\put(4.5,1){\circle*{0.1}}
\put(4.5,2){\circle*{0.1}}
\put(4.5,3){\circle*{0.1}}
\put(4.7,1){\circle*{0.1}}
\put(4.7,2){\circle*{0.1}}
\put(4.7,3){\circle*{0.1}}
\put(4.7,4){\circle*{0.1}}
\put(4.7,4.5){\makebox(0,0){$\vdots$}}
\put(4.9,1){\circle*{0.1}}
\put(5.1,1){\circle*{0.1}}

\put(3.9,4.7){$d$ even, $r =4$}

\put(1.5,1){\line(0,1){3}}
\put(2.5,1){\line(1,1){2}}
\put(4.5,1){\line(0,1){2}}
\put(4.7,1){\line(0,1){3}}
\put(0,1){\makebox(0,0){$0$}}
\put(0,2){\makebox(0,0){$1$}}
\put(0,3){\makebox(0,0){$2$}}
\put(0,4.5){\makebox(0,0){$s$}}
\put(1.5,0){\makebox(0,0){$d-3$}}
\put(2.5,0){\makebox(0,0){$d-2$}}
\put(3.5,0){\makebox(0,0){$d-1$}}
\put(4.5,0){\makebox(0,0){$d$}}

\put(6,0){\makebox(0,0){$t-s$}}

\put(7.5,0.5){\vector(0,1){5}}
\put(7.5,0.5){\vector(1,0){6}}
\put(8.5,1){\circle*{0.1}}

\put(9.5,1){\circle*{0.1}}

\put(10.5,1){\circle*{0.1}}
\put(10.5,2){\circle*{0.1}}
\put(10.7,1){\circle*{0.1}}

\put(11.5,1){\circle*{0.1}}
\put(11.5,2){\circle*{0.1}}
\put(11.5,3){\circle*{0.1}}
\put(11.7,1){\circle*{0.1}}
\put(11.9,1){\circle*{0.1}}

\put(10.9,4.7){$d $ odd, $r=4$}
\put(9.5,1){\line(1,1){2}}
\put(11.5,1){\line(0,1){2}}
\put(13,0){\makebox(0,0){$t-s$}}
\put(7,1){\makebox(0,0){$0$}}
\put(7,2){\makebox(0,0){$1$}}
\put(7,3){\makebox(0,0){$2$}}
\put(7,4.5){\makebox(0,0){$s$}}
\put(8.5,0){\makebox(0,0){$d-3$}}
\put(9.5,0){\makebox(0,0){$d-2$}}
\put(10.5,0){\makebox(0,0){$d-1$}}
\put(11.5,0){\makebox(0,0){$d$}}

\end{picture}
\end{equation*}
\caption{The Adams spectral sequence, the unoriented case.} \label{Assunor}
 \end{figure}

\begin{rem}
Note how the $E_2$-terms seem to depend only on $d$ mod $2$ even though the computations depend on $d$ mod $4$. Similarly, the spectral sequences for $\pi_d(MTSO(d,5))$ only depended on $d$ mod $ 4$ even though the computations depended on $d$ mod $8$. We do no know whether this is part of a general phenomenon.  The periodicity map only explains the 4- and 8-periodicity, respectively. 
\end{rem}

\subsection{The spin case}\label{spin}
The cohomology of $BSpin(d) $ is more complicated. The map $BSpin (d) \to BSO(d)$ induces a map
\begin{equation*}
H^*(BSO(d);\Z/2) \to H^*(BSpin(d);\Z/2).
\end{equation*}
Let $J_d$ denote the ideal in $H^*(BSO(d);\Z/2)$ generated by
\begin{equation}\label{Jgen}
w_2, Sq^1(w_2), Sq^2 Sq^1(w_2),\dots , Sq^{\frac{1}{2}a_d} Sq^{\frac{1}{4}a_d} \dotsm Sq^1(w_2).
\end{equation}
Here $a_d$ is the power of $2$ given in the table \eqref{RH}.
Quillen showed in \cite{quillen}:

\begin{thm}
The kernel of $H^*(BSO(d);\Z/2) \to H^*(BSpin(d);\Z/2)$ is exactly $J_d$ and 
\begin{equation*}
H^*(BSpin(d);\Z/2) \cong H^*(BSO(d);\Z/2)/J_d \oplus \Z/2[v_{a_d}]. 
\end{equation*}
Here $v_{a_d}$ is a class in dimension $a_d$.
\end{thm}

\begin{thm}\label{spincoho}
$H^q(BSpin(d),BSpin(d-r);\Z/2)$  is isomorphic to the kernel of
\begin{equation*}
 H^q(BSpin(d);\Z/2) \to H^q(BSpin(d-r);\Z/2)
\end{equation*}
for $q < a_{d-r} $, and the map 
\begin{equation*}
p^*: H^q(BSO(d),BSO(d-r);\Z/2) \to H^q(BSpin(d),BSpin(d-r);\Z/2) 
\end{equation*}
is surjective with kernel $J_d \cap H^q(BSO(d),BSO(d-r);\Z/2)$. 
\end{thm}

\begin{proof}
In the following, coefficients in $\Z/2$ are understood.
There is a map of long exact sequences
\begin{equation*}
\xymatrix{{H^*(BSO(d),BSO(d-r))}\ar[d]^{p^*} \ar[r]^-{j^*}&{H^*(BSO(d))}\ar[r]^-{i^*} \ar[d] &{H^*(BSO(d-r))} \ar[d]^{p^{\prime *}}\\
{H^*(BSpin(d),BSpin(d-r))}\ar[r]^-{j^{\prime*}}&{H^*(BSpin(d))}\ar[r]^-{i^{\prime*}}&{H^*(BSpin(d-r)).}
}
\end{equation*}

Since both $i^*$ and $p^{\prime *}$ are surjective in dimensions less than $a_{d-r}$, we see that $i^{\prime *}$ is also surjective, so $H^q(BSpin(d),BSpin(d-r))$ is just the kernel of $i^{\prime *}$ when $q \leq a_{d-r} $.

Thus we just need to describe 
\begin{equation}\label{kernelJ2}
\Ker(\Z/2[w_2,\dots ,w_d]/J_d \to \Z/2[w_2,\dots ,w_{d}]/J'_{d-r})
\end{equation}
where $J_{d-r}'$ is the ideal generated by $J_{d-r}$ and $w_{d-r+1},\dots ,w_d$. We must show that this is
\begin{equation}\label{kernelJ}
\Ker(\Z/2[w_2,\dots ,w_d] \to \Z/2[w_2,\dots ,w_{d-r}])/J
\end{equation}
where $J=J_{d}\cap \Ker(\Z/2[w_2,\dots ,w_d] \to \Z/2[w_2,\dots ,w_{d-r}])$.

Clearly, there is an injective map from \eqref{kernelJ} to \eqref{kernelJ2}. Now, let $P$ be some polynomial in the Stiefel--Whitney classes representing an element in \eqref{kernelJ2}. Then $P \in J_{d-r}'$. Let $g_1,\dots ,g_m$ denote the generators of $J_{d}$ in dimensions up to $a_{d-r}$ given by \eqref{Jgen} and $g_1',\dots ,g_m'$ the generators of $J_{d-r}$ given by \eqref{Jgen}. Under the map $H^*(BSO(d)) \to H^*(BSO(d-r))$, $g_i$ maps to $g_i'$ by naturality of the Steenrod squares. Thus, $g_i$ and $g_i'$ differ only by an element of the ideal generated by $w_{d-r+1},\dots ,w_d$. This means that $\{g_1,\dots ,g_m,w_{d-r+1},\dots ,w_d\}$ is also a set of generators for $J'_{d-r}$. So for suitable polynomials $\lambda_1,\dots ,\lambda_m$ and $\mu_1,\dots ,\mu_r$, 
\begin{equation*}
P = \lambda_1 g_1+\dots +\lambda_m g_m +\mu_1 w_{d-r+1}+\dots + \mu_r w_d.
\end{equation*}
But $\lambda_1g_1+\dots +\lambda_m g_m \in J_{d}$ so $P$ represents the same element as $\mu_1 w_{d-r+1}+\dots + \mu_r w_d$ in $\Z/2[w_2,\dots ,w_d]/J_d$. This lies in \eqref{kernelJ}, so the map is also surjective.
\end{proof}

\begin{cor}\label{spinfree}
Assume $2r < d$ and $9\leq d-r$.
In dimensions $*<2(d-r)$, $H^*(MTSpin(d,r);\Z/2)$ is isomorphic to the free $H^*(BSpin;\Z/2)$-module on generators
\begin{equation*}
p^*(w_{d-r+1}),\dots , p^*(w_d)
\end{equation*}
where $p^*$ is as in Theorem \ref{spincoho}.
\end{cor}

\begin{proof}
By Theorem \ref{spincoho}, it is enough to show that
\begin{equation*}
J_d \cap H^q(BSO(d),BSO(d-r);\Z/2) \subseteq  J_d \cdot H^*(BSO(d),BSO(d-r);\Z/2)
\end{equation*}
for all $q\leq 2(d-r)$.
Let $g_k$ denote the generator of $J_d$ of degree $2^{k}+1$ given in \eqref{Jgen}. Suppose 
\begin{equation*}
\lambda_0 g_0  + \dotsm +\lambda_m g_m \in J_d\cap H^q(BSO(d),BSO(d-r);\Z/2)
\end{equation*}
where $m$ is unique with the property $d-r<2^m +1 < 2(d-r)$. Then the image in $H^q(BSO(d-r);\Z/2)$ must be zero. But this is also $\sum_k \lambda_k' g_k' \in J_{d-r}$ where the $g_k'$ are the generators of $J_{d-r}$ and $\lambda_k'$ is the image of $\lambda_k$. The $g_k'$ form a regular sequence in the sense of \cite{quillen}. Hence we deduce that $\lambda_m'$ belongs to the ideal generated by $g_{0}',\dots,g_{m-1}'$. Since $\deg(\lambda_m)<d-r$, also $\lambda_m$ lies in the ideal generated by $g_{0},\dots,g_{m-1}$. Rearranging the terms, we may assume $\lambda_m=0$. But none of the $g_k$ with $k<m$ contains terms involving $w_{d-r+1},\dots ,w_{d}$ for degree reasons. Thus we can choose all $\lambda_{0}, \dots,\lambda_{m-1}$  in $H^*(BSO(d),BSO(d-r);\Z/2)$. This proves the claim. 
\end{proof}

\begin{thm}
Let $F$ be either $\Q$ or $\Z/p$ for $p$ an odd prime. Then
\begin{equation*}
H^*(MTSO(d,r);F) \to  H^*(MTSpin(d,r);F)
\end{equation*}
is an isomorphism.
\end{thm}

\begin{proof}
Consider the fibration
\begin{equation*}
B\Z/2 \to BSpin(d) \to BSO(d).
\end{equation*}
Since $\widetilde{H}^*(B\Z/2,F) =0$, it follows from the Serre spectral sequence that 
\begin{equation*}
H^*(BSO(d);F) \to  H^*(BSpin(d);F) 
\end{equation*}
is an isomorphism. By the 5-lemma,
\begin{equation*}
H^*(BSO(d),BSO(d-r);F) \to  H^*(BSpin(d),BSpin(d-r);F)
\end{equation*}
is an isomorphism. A Thom isomorphism yields the result.
\end{proof}

\begin{cor}\label{spinper}
The periodicity map $\Sigma^{a_r} MTSpin(d,r)\to MTSpin(d+a_r,r)$ is a $(2(d-r)+a_r-1)$-equivalence for $2r<d$ and $9\leq d-r$.
\end{cor}

\begin{proof}
It is enough to see that
\begin{equation*}
H^{q+a_r}(MTSpin(d+a_r,r);\Z/2)\to H^q(MTSpin(d,r);\Z/2)
\end{equation*}
is an isomorphism for $q<2(d-r)$.
This follows because the periodicity map for $MTSO(d,r)$ maps $J_{d+a_r}\cdot H^*(MTSO(d+a_r,r);\Z/2)$ to $J_{d}\cdot H^*(MTSO(d,r);\Z/2)$.
\end{proof}

We also need some information about the $\mathcal{A}_3$-action on $H^*(MTSO(d);\Z/3)$. Let $\mathcal{P}^1$ denote the first Steenrod power. It is shown in \cite{borel} that
\begin{equation*}
\mathcal{P}^1(c_{j-2}) = c_1^2c_{j-1} - 2c_2c_{j-2} - c_1c_{j-1} + jc_j
\end{equation*}
where $c_i \in H^{2i}(BU(d);\Z/3)$ is the $i$th Chern class. Since $p_j = g^*((-1)^jc_{2j})$ under the map $g:BSO(d) \to BU(d)$, we get
\begin{equation}\label{pont}
\mathcal{P}^1(p_{j-1}) = 2p_1p_{j-1} - 2jp_{j}.
\end{equation}
See \cite{milnors} for the relation between Chern and Pontryagin classes.

The above considerations in cohomology allow us to calculate the homotopy groups $\pi_{q}(MTSpin(d,r))$. We will only consider dimensions $q\leq d-r+9$ and assume $q<a_{d-r}$. Then the cohomology with $\Z/2$ coefficients is just $H^q(MTSO(d,r);\Z/2)$ with the relations 
\begin{equation}\label{rela}
\begin{split}
&w_2=w_3=w_5 = w_9 = 0\\
&w_{17} + w_4w_{13} + w_6w_{11} + w_7w_{10} =0 
\end{split}
\end{equation}
as the only relevant ones.

In the following table, a basis of $H^q(MTSpin(d,r);\Z/2)$ is shown for low values of $q$.  Of course, when $d-r+k>d$, then $w_{d-r+k}=0$ is understood. To avoid the relations \eqref{rela}, we will assume $d-r+1>9$, and when we deal with cohomology in dimensions  $q>d-r+7$, we also assume $q-7>10$.
\begin{equation*}
\begin{tabular}{|c|c|ccccc|}
\hline
{$q$}&{$H^q(V_{d,r})$}&\multicolumn{5}{|c|}{$H^q(C_\theta )$}\\
\hline 
{$d-r+1$}&{$w_{d-r+1}$}&{}&{}&{}&{}&{}\\
{$d-r+2$}&{$w_{d-r+2}$}&{}&{}&{}&{}&{}\\
{$d-r+3$}&{$w_{d-r+3}$}&{}&{}&{}&{}&{}\\
{$d-r+4$}&{$w_{d-r+4}$}&{}&{}&{}&{}&{}\\
{$d-r+5$}&{$w_{d-r+5}$}&{$w_4w_{d-r+1}$}&{}&{}&{}&{}\\
{$d-r+6$}&{$w_{d-r+6}$}&{$w_4w_{d-r+2}$}&{}&{}&{}&{}\\
{$d-r+7$}&{$w_{d-r+7}$}&{$w_4w_{d-r+3}$}&{$w_6w_{d-r+1}$}&{}&{}&{}\\
{$d-r+8$}&{$w_{d-r+8}$}&{$w_4w_{d-r+4}$}&{$w_6w_{d-r+2}$}&{$w_7w_{d-r+1}$}&{}&{}\\
{$d-r+9$}&{$w_{d-r+9}$}&{$w_4w_{d-r+5}$}&{$w_6w_{d-r+3}$}&{$w_7w_{d-r+2}$}&{$w_8w_{d-r+1}$}&{$w_4^2w_{d-r+1}$}\\
\hline
\end{tabular}
\end{equation*}
Projection onto the first column yields $H^*(V_{d,r};\Z/2)$, while the second column is the cohomology of the cofiber $C_\theta$ of the inclusion $\Sigma^\infty \Sigma V_{d,r} \to MTSpin(d,r)$. This follows because we know the composite map
\begin{equation*}
H^*(MTSO(d,r);\Z/2) \to H^*(MTSpin(d,r);\Z/2) \to H^*(\Sigma V_{d,r};\Z/2)
\end{equation*}
is surjective, so $H^*(C_\theta;\Z/2)$ is the kernel of the last map.

We may now determine the action of the Steenrod algebra from the action on $H^*(MTSO(d,r);\Z/2)$.
In dimensions up to $d-r+8$, the Steenrod action respects the two columns. However, $Sq^8(w_{d-r+1})$ contains mixed terms. Applying the Adams spectral sequence, we get:

\begin{thm}
For $q<2d$, $d>9$, and $q<a_{d-1}$, $\pi_{q}(MTSpin(d,1))$ is given by the table
\begin{equation*}
\begin{tabular}{ccccccccc}
\hline
 {q}&{d}&{d+1}&{d+2}&{d+3}&{d+4}&{d+5}&{d+6}&{d+7} \\
\hline
 {$\pi_q^s(S^d)$}&{$\Z$}&{$\Z$/2}&{$\Z$ /2}&{$\Z$ /24} &{0}&{0} &{$\Z/2$}&{$\Z/240$}\\

 {$\pi_q^s(\Sigma^d BSpin(d))$}&{0}&{0}&{0}&{0} &{$\Z $} & {0} &{0} &{0}\\
\hline
\end{tabular}
\end{equation*}
For the last two columns, we must assume $q>17$.
\end{thm}

\begin{proof}
This is exactly similar to the oriented case. To calculate the 3-torsion, note that 
\begin{equation*}
H^{d+4}(\Sigma^{d+\infty} BSpin(d);\Z/3) \cong H^{4}(BSO(d);\Z/3) \cong \Z/3,
\end{equation*}
generated by the first Pontryagin class $p_1$. But 
\begin{equation*}
\mathcal{P}^1(p_1) = 2p_1^2 -p_2 \neq 0
\end{equation*}
by \eqref{pont}. The Adams spectral sequence then shows that there is no 3-torsion in $\pi_{d+7}^s(\Sigma^d BSpin(d))$.
\end{proof}

\begin{thm}
$\pi_q(C_\theta)$ is given by the following table for $r>4$, $q < 2(d-r)-1$, and $d-r>10$: 
\begin{equation*}
\begin{tabular}{cccccc}
\hline
 {$q$}&{$q \leq d-r+4$}&{$d-r+5$}&{$d-r+6$}&{$d-r+7$}&{$d-r+8$} \\
\hline
 {$d-r$ even}&{0}&{$\Z$}&{$\Z/2$}&{$0$}&{$\Z/2$} \\
 {$d-r$ odd}&{0}&{$\Z/2$}&{$0$} &{$\Z/2$}&{0}\\
\hline
\end{tabular}
\end{equation*}
\end{thm}

\begin{proof}
Again this follows from the spectral sequence. For $d-r$ even, there could be some 3-torsion since there is a $\phi(\delta(e_{d-r}p_1)) \in H^{d-r+5}(C_\theta ; \Z/3)$. But under the composite map $g: MTSpin(d-r+1,1) \to MTSpin(d,r) \to C_\theta $,
\begin{equation*}
 g^*(\mathcal{P}^1(\delta(e_{d-r}p_1))) =  \mathcal{P}^1(p_1) = 2p_1^2 - p_2 \neq 0
\end{equation*}
by the $r=1$ case. Using this, the Adams spectral sequence shows that there is no $3$-torsion.
\end{proof}

\begin{cor}\label{spininj}
When $d-r>10$ is odd and $q < 2(d-r) - 1$,
\begin{equation*}
\theta^r : \pi_{q-1}(V_{d,r}) \to \pi_q(MTSpin(d,r))
\end{equation*}
is an isomorphism for $q = d-r+6$, and it is injective for $q=d-r+7$. In particular, 
\begin{equation*}
\theta^7 : \pi_{d-1}(V_{d,7}) \to \pi_d(MTSpin(d,7))
\end{equation*}
is injective for  $d$ even.
\end{cor}

According to Theorem \ref{theta-inj}, $\pi_{q-1}(V_{d,r}) \to \pi_q(MTSpin(d,r))$ is an isomorphism for $q\leq d-r+4$, so we know the homotopy groups in these dimensions. Here are some more:

\begin{thm}
$\pi_{q}(MTSpin(d,r))$ is given  by
\begin{equation*}
\begin{tabular}{cccc}
\hline
{$r$}&{$d$}&{$q$}&{$\pi_{q}(MTSpin(d,r))$}\\
\hline
{$5,6$}&{$0 \bmod 4$}&{$d$}&{$\Z \oplus \Z/8 \oplus \Z/2$}\\
{$5$}&{$1 \bmod 4$}&{$d$}&{$\Z \oplus \Z/2 \oplus \Z/2$}\\
{$5,6$}&{$2 \bmod 4$}&{$d$}&{$\Z \oplus \Z/2$}\\
{$5$}&{$3 \bmod 4$}&{$d$}&{$\Z  \oplus \Z/2$}\\
{$6$}&{$1 \bmod 4$}&{$d$}&{$\Z/2 \oplus \Z/2$}\\
{$6$}&{$3 \bmod 4$}&{$d$}&{$\Z/2$}\\
\hline
\end{tabular}
\end{equation*}
for $q<2(d-r)-1$, $9\leq d-r$, and $q<a_{d-r}$.
\end{thm}

\section{Vector field cobordism groups}\label{calcabs}
The spectra $MT(d)$ introduced in Section \ref{def1} were linked to cobordism theory in the paper \cite{GMTW}.
In \cite{ab}, the groups $\pi_{d+r}(MTSO(d))$ were identified as the vector field cobordism groups. A vector field cobordism between two manifolds $M$ and $N$ with $r$ independent sections given in $TM\oplus \R$ and $TN\oplus \R$, respectively, is a cobordism with $r$ independent tangent vector fields extending the ones given on the boundary, the $\R$ direction corresponding to the inward normal at $M$ and the outward normal at $N$. For $r<\frac{d}{2}$, this is an equivalence relation and $\pi_{d+r}(MTSO(d))$ is the set of equivalence classes.


%

In this interpretation,
\begin{equation}\label{glem}
\pi_{d+r}(MTSO(d)) \to \pi_{d+r}(MTSO(d+r+1)) \cong \Omega_{d+r}
\end{equation}
is just the map that forgets the vector fields. Here $\Omega_{d+r}$ is the usual oriented cobordism group.

The purpose of Section \ref{VFCG} is to apply the computations of Section \ref{cal} to determine the groups $\pi_{d+r}(MTSO(d))$. 
Since $\Omega_{d+r}$ is well-known, we want to describe $\pi_{d+r}(MTSO(d))$ in terms of the map~\eqref{glem}. 
The image of \eqref{glem} is the set of cobordism classes containing a manifold with $r$ independent tangent vector fields, while the kernel is the subgroup of classes containing bounding manifolds, see \cite{ab}. 
For the remainder of this chapter, we shall abbreviate $MTSO$  to $MT$.

%
%
The idea is to look at the long exact sequence for the cofibration 
\begin{equation*}
MT(d) \to MT(d+r) \to MT(d+r,r)
\end{equation*}
 in order to determine $\pi_{d+r-1}(MT(d))$:
\begin{equation} \label{flg}
\begin{tabular}{c@{\ }c@{\ }c@{\ }c@{\ }c@{\ }c}
&$\dotsm $&$\to $&$\pi_{d+r}(MT(d+r))$&$\xrightarrow{\beta^r} $&$\pi_{d+r}(MT(d+r,r))$ \\
$\to $&$\pi_{d+r-1}(MT(d)) $&$\to $&$\Omega_{d+r-1} $&$\to $&$\pi_{d+r-1}(MT(d+r,r)) $ \\
&$\dotsm $& &$\dotsm $ & &$\dotsm $ \\
$\to $&$\pi_{d+1}(MT(d)) $&$\to $&$\Omega_{d+1} $&$\to $&$\pi_{d+1}(MT(d+r,r)) $ \\
$\to $&$\pi_{d}(MT(d))$ &$\to$ &$\Omega_d $&$\to $&$0$
\end{tabular}
\end{equation}
Here $\beta^r$ is the map $\beta(M) \mapsto \beta^r(M)$ that was also denoted by $j_{r*}$ in \eqref{jrdef}.

The maps to the right of $\pi_{d+r-1}(MT(d))$ are easily described by induction on $r$ and the computations of Section \ref{cal}. Thus the main problem is to determine the image of $\beta^r$. For this, recall that we have the map
\begin{equation}\label{bpsi}
\pi_{d+r}(MT(d+r)) \xrightarrow{\beta^r} \pi_{d+r}(MT(d+r,r)) \xrightarrow{\Psi} KR^{t-d}(tH_r).
\end{equation}
 It follows from Theorem \ref{KRfak} and \cite{AD} that $\Psi$ is an isomorphism for $r=1,2$ and a surjection for $r=3$. Every element in $\pi_{d+r}(MT(d+r))$ is $\beta(M) $ for some compact oriented  $(d+r)$-manifold $M$, and $\Psi \circ \beta^r$ is the Atiyah--Dupont invariant of this $M$. 
 
From these computations, we shall obtain a description of the global invariant $\beta^4(M)$ in Section \ref{ident}.

\subsection{The computations}\label{VFCG}
For $r=1$, consider the map $MT(d) \to MT(d+1)$. There is a diagram with exact rows
\begin{equation}\label{diag} 
\vcenter{\xymatrix{{\pi_{d+1}(MT(d+1,1))}\ar[d]\ar[r]&{\pi_d(MT(d))}\ar[r]\ar[d]^{\beta^1}&{\pi_d(MT(d+1))}\ar[r] \ar[d]&{0}\\
{\pi_{d+1}(MT(d+1,1))}\ar[r]&{\pi_d(MT(d,1))}\ar[r]&{\pi_d(MT(d+1,2))}\ar[r]&{0.}
} }
\end{equation}
Here $\beta^1(M)=\chi(M)$ by Theorem \ref{beta=ind}. From this, $\pi_d(MT(d))$ can be computed.
This is done in e.g.\ \cite{GMTW} or \cite{ebert}. We include the statements and proofs here for completeness and later reference. 

\begin{prop}\label{seseven}
Let $d$ be even. There is a split short exact sequence 
\begin{equation}\label{pi1even}
0 \to \Z \to \pi_d(MT(d)) \to \Omega_d \to 0.
\end{equation}
A splitting $\pi_d(MT(d)) \to \Z$ is given by $\frac{1}{2}\chi $ when $d \equiv 2 \bmod 4$ and $ \frac{1}{2}
(\sigma+ \chi)$ when $d \equiv 0 \bmod 4$. Here $\sigma $ is the signature.
\end{prop}

\begin{proof}
If we fill in the known groups, \eqref{diag} becomes
\begin{equation*} 
\xymatrix{{\Z}\ar[d]^{\cong}\ar[r]&{\pi_d(MT(d))}\ar[r]\ar[d]^{\chi}&{\Omega_d}\ar[r]\ar[d]&{0}\\
{\Z}\ar[r]^{\cdot 2}&{\Z}\ar[r]&{\Z/2}\ar[r]&{0.}
}
\end{equation*}
For $d \equiv 2 \bmod 4$, all manifolds have even Euler characteristic, so by the diagram, $\frac{1}{2}\chi$ defines a splitting of the upper row in the diagram.

For $d \equiv 0 \bmod 4$, $\chi$ is surjective. For instance, $\chi(S^d)=2$ and $\chi(\C P^{2n})=n+1$. Now, 
\begin{equation*}
\begin{split}
\sigma &: \pi_d(MT(d)) \to \Omega_d \to \Z\\
\chi &: \pi_d(MT(d)) \to \pi_d(MT(d,1)) \cong \Z
\end{split}
\end{equation*}
are both well-defined homomorphisms and they always have the same parity. Therefore, $\frac{1}{2}(\sigma + \chi) : \pi_d(MT(d)) \to \Z$ is well-defined. This clearly defines a splitting.
\end{proof}

\begin{prop}\label{sesodd}
For $d \equiv 3 \bmod 4$, $\pi_d(MT(d)) \cong \Omega_d$. When $d \equiv 1 \bmod 4$, there is a short exact sequence
\begin{equation}\label{pi1odd}
0 \to \Z/2 \to \pi_d(MT(d)) \to \Omega_d \to 0
\end{equation}
which is split by the real semi-characteristic $\chi_{\R}$, defined in \eqref{chi2}.
\end{prop}

\begin{proof}
The long exact sequence for the pair $(MT(d),MT(d+1))$ is
\begin{equation*}
\pi_{d+1}(MT(d+1)) \xrightarrow{\chi} \pi_{d+1}(MT(d+1,1))\to \pi_d(MT(d)) \to \Omega_d \to 0.
\end{equation*}
Here $\chi$ is surjective when $d+1 \equiv 0 \bmod 4$ and has image $2\Z$ when $d+1 \equiv 2 \bmod 4$. This yields the isomorphism and the short exact sequence. To see that the sequence splits, consider the diagram
\begin{equation*} 
\xymatrix{{\pi_{d+1}(MT(d+1,1))}\ar[d]\ar[r]&{\pi_d(MT(d))}\ar[r]\ar[d]&{\pi_d(MT(d+1))}\ar[r] \ar[d]&{0}\\
{\pi_{d+1}(MT(d+1,1))}\ar[r]&{\pi_d(MT(d,2))}\ar[r]&{\pi_d(MT(d+1,3))}\ar[r]&{0.}
}
\end{equation*}
With the known groups inserted, this becomes
\begin{equation*} 
\xymatrix{{\Z}\ar[d]^{\cong}\ar[r]&{\pi_d(MT(d))}\ar[r]\ar[d]^{\beta^2}&{\Omega_d}\ar[r] \ar[d]&{0}\\
{\Z}\ar[r]&{\Z/2}\ar[r]&{0}\ar[r]&{0.}
}
\end{equation*}
Thus $\beta^2$ defines a splitting. By \eqref{bpsi} and \cite{AD}, $\beta^2(M)=\chi_\R(M)$.
\end{proof}

We now proceed to the higher homotopy groups $\pi_{d+r}(MT(d))$.
\begin{thm}\label{pi1MTd4}
For $d \equiv 0 \bmod 4$, there is a split short exact sequence
\begin{equation*}
0 \to \Z/2 \oplus \Z/2 \to \pi_{d+1}(MT(d)) \to \Omega_{d+1} \to 0.
\end{equation*}
One of the $\Z/2$ summands is $\chi_\R$, while the other one is the generator of the $\Z$ in~\eqref{pi1even} composed with a Hopf map. 
\end{thm}

\begin{proof}
By Theorem \ref{MTd2}, the sequence \eqref{flg} is as follows:
\begin{equation*}
\begin{tabular}{c@{\ }c@{\ }c@{\ }c@{\ }c@{\ }c}
 {} & $\dotsm $  & $\to $& $\pi_{d+2}(MT(d+2))$ & $\xrightarrow{\beta^2}$ & $\Z \oplus \Z/2$ \\
$\to$ & $\pi_{d+1}(MT(d)) $& $\to $& $\Omega_{d+1}$ & $\to$ & $\Z$ \\
$\to$ & $\pi_{d}(MT(d))$ & $\to$ & $\Omega_d $ & $\to $&  $0$\\
\end{tabular}
\end{equation*}
This only requires $d\geq 4$. It follows from Proposition \ref{seseven}  that
\begin{equation*}
\pi_{d+1}(MT(d)) \to \Omega_{d+1}
\end{equation*}
is surjective.

By \eqref{bpsi} and \cite{AD},  $\beta^2(M^{d+2}) = (\chi (M), 0)$. Since $d+2 \equiv 2 \bmod 4$, the Euler characteristic is even, so the cokernel of $\beta^2$ is $\Z/2 \oplus \Z/2$, i.e.\ we get the short exact sequence
\begin{equation*}
0\to \Z/2 \oplus \Z/2 \to \pi_{d+1}(MT(d)) \to \Omega_{d+1} \to 0.
\end{equation*}

This together with Proposition \ref{sesodd} means that the long exact sequence
\begin{equation*}
\dotsm \to \pi_{d+2}(MT(d+1,1)) \to \pi_{d+1}(MT(d)) \to \pi_{d+1}(MT(d+1)) \to \dotsm
\end{equation*}
becomes a short exact sequence
\begin{equation}\label{kef}
0 \to \Z/2 \to \pi_{d+1}(MT(d)) \to \Omega_{d+1} \oplus \Z/2 \to 0.
\end{equation}
Together with the fact that $\pi_{d+2}(MT(d+1,1))\cong \Z/2$ is generated by the composition of the generator in $\pi_{d+1}(MT(d+1,1))\cong \Z$ with a Hopf map, this yields the interpretation of the $\Z/2 \oplus \Z/2$.

To see that \eqref{kef} splits, consider the diagram with exact rows
\begin{equation*} 
\vcenter{\xymatrix{{}\ar[r]&{\pi_{d+2}(MT(d+1,1))}\ar[d]\ar[r]&{\pi_{d+1}(MT(d))}\ar[r]\ar[d]&{\pi_{d+1}(MT(d+1))}\ar[r] \ar[d]&{}\\
{}\ar[r]&{\pi_{d+2}(MT(d+1,1))}\ar[r]&{\pi_{d+1}(MT(d,2))}\ar[r]&{\pi_{d+1}(MT(d+1,3))}\ar[r]&{.}
}}
\end{equation*}
Inserting the known groups, we get 
\begin{equation*} 
\xymatrix{{}\ar[r]&{\Z/2} \ar[d]^\cong \ar[r]&{\pi_{d+1}(MT(d))}\ar[r]\ar[d]&{\pi_{d+1}(MT(d+1))}\ar[r]\ar[d]&{}\\
{}\ar[r]&{\Z/2}\ar[r]&{\Z/2 \oplus \Z/2 \oplus \Z/2}\ar[r]&{\Z/2 \oplus \Z/2}\ar[r]&{.}
}
\end{equation*}
The calculation of these homotopy groups only requires $d\geq 4$. It follows from the diagram that the upper sequence must split.

For $d=0$, the theorem follows by direct computations.
\end{proof}

\begin{thm}
For $d \equiv 1 \bmod 4$, there is a split short exact sequence
\begin{equation*}
0 \to \Z/2 \to \pi_{d+1}(MT(d)) \to \Omega_{d+1} \to 0.
\end{equation*}
The $\Z/2$ summand is the $\Z/2$ from \eqref{pi1odd} composed with a Hopf map.
\end{thm}

\begin{proof}
In this case, the sequence \eqref{flg} becomes:
\begin{equation*}
\begin{tabular}{c@{\ }c@{\ }c@{\ }c@{\ }c@{\ }c}
 &$\dotsm $ & $\to$ & $\pi_{d+2}(MT(d+2))$&$\xrightarrow{\beta^2}$ &$\Z/2$ \\
$\to$& $\pi_{d+1}(MT(d))$& $\to$& $\Omega_{d+1}$& $\to$ &$\Z/2$ \\
$\to$& $\pi_{d}(MT(d))$& $\to$ & $\Omega_d$ & $\to$ &$0$
\end{tabular}
\end{equation*}
All manifolds of dimension $d+2 \equiv 3 \bmod 4$ have two independent vector fields, so $\beta^2(M)=0$ for all $M$. Together with \eqref{pi1odd},  this yields a short exact sequence
\begin{equation*}
0 \to \Z/2 \to \pi_{d+1}(MT(d)) \to \Omega_{d+1} \to 0 
\end{equation*}
for all $d>0$. Since $\pi_{d+2}(MT(d+2,2))\cong \eta \cdot \pi_{d+1}(MT(d+2,2))$ where $\eta$ is the Hopf map, the interpretation of the $\Z/2$ immediately follows.

To solve the extension problem, consider the diagram
\begin{equation*} 
\xymatrix{{}\ar[r]&{\pi_{d+2}(MT(d+1,1))}\ar[d]\ar[r]&{\pi_{d+1}(MT(d))}\ar[r]\ar[d]&{\pi_{d+1}(MT(d+1))}\ar[r] \ar[d]^{\cong}&{}\\
{}\ar[r]&{\pi_{d+2}(MT(d+1,1))}\ar[r]&{\pi_{d+1}(MT(d,3))}\ar[r]&{\pi_{d+1}(MT(d,4))}\ar[r]&{.}
}
\end{equation*}
For $d>6$, the groups are
\begin{equation*} 
\xymatrix{{\Z/2}\ar[d]^{\cong}\ar[r]&{\pi_{d+1}(MT(d))}\ar[r]\ar[d]&{\pi_{d+1}(MT(d+1))}\ar[r] \ar[d]&{\Z}\\
{\Z/2}\ar[r]&{\Z/24 \oplus \Z/2}\ar[r]&{\Z \oplus \Z/24}\ar[r]&{\Z.}
}
\end{equation*}
The first map in the lower row must be the inclusion of a direct summand, and thus the sequence splits.

For $d=1$ and $d=5$,
$\Omega_2=\Omega_6=0$
so the extension problem is trivial.
\end{proof}

\begin{thm}
For $d \equiv 2 \bmod 4$, $\pi_{d+1}(MT(d)) \cong \Omega_{d+1}$. 
\end{thm}
\begin{proof}
In this case, the exact sequence is:
\begin{equation*}
\begin{tabular}{c@{\ }c@{\ }c@{\ }c@{\ }c@{\ }c}
&$\dots$& $\to $&$\pi_{d+2}(MT(d+2))$&$\xrightarrow{\beta^2}$&$ \Z \oplus \Z/2$ \\
 $\to$& $\pi_{d+1}(MT(d))$& $\to $&$\Omega_{d+1}$& $\to$& $\Z$ \\
$\to$&$\pi_{d}(MT(d))$ &$\to$& $\Omega_d$& $\to$& $0$
\end{tabular}
\end{equation*}
By \eqref{bpsi} and \cite{AD}, 
\begin{equation*}
\beta^2(M) = \left(\chi(M),\frac{1}{2}(\chi(M)+\sigma(M))\right)\in \Z \oplus \Z/2
\end{equation*}
when $d>2$.
This is surjective because
\begin{equation}\label{chi+sigma}
\left(\chi,\frac{1}{2}(\chi+\sigma)\right) : \pi_{d+2}(MT(d+2)) \to \Z \oplus \Z
\end{equation}
is surjective by Proposition \ref{seseven} and the fact that $\sigma: \Omega_{d+2} \to \Z$ is surjective. 

When $d=2$, the claim follows by a direct computation.
\end{proof}

\begin{thm}
For $d \equiv 3 \bmod 4$, there is a short exact sequence
\begin{equation*}
0 \to \pi_{d+1}(MT(d)) \to \Omega_{d+1} \to \Z/2 \to 0.
\end{equation*}
The map $\Omega_d \to \Z/2$ is $\chi$ (or equivalently $\sigma$) mod $2$.
\end{thm}

\begin{proof}
In this case, the exact sequence becomes:
\begin{equation*}
\begin{tabular}{c@{\ }c@{\ }c@{\ }c@{\ }c@{\ }c}
&$\dotsm$& $\to$ &$\pi_{d+2}(MT(d+2))$&$\xrightarrow{\beta^2}$&$ \Z/2$ \\
$\to$&$\pi_{d+1}(MT(d))$& $\to$ & $\Omega_{d+1}$& $\to$& $\Z/2$  \\
$\to$&$\pi_{d}(MT(d))$ &$\to$& $\Omega_d$& $\to$ & $0$
\end{tabular}
\end{equation*}
By \eqref{bpsi} and \cite{AD}, $\beta^2(M) = \chi_{\R}(M)$ so $\beta^2$ is surjective. Thus we get a short exact sequence
\begin{equation*}
0\to \pi_{d+1}(MT(d)) \to \Omega_{d+1} \to \Z/2 \to 0.
\end{equation*}
The last map is the Euler characteristic mod $2$ because the diagram
\begin{equation*} 
\xymatrix{{\pi_{d+1}(MT(d))}\ar[r]&{\pi_{d+1}(MT(d+1))}\ar[r]^-{\chi}\ar[d]&{\pi_{d+1}(MT(d+1,1))\cong \Z}\ar[d]\\
{}&{\pi_{d+1}(MT(d+2))}\ar[r]&{\pi_{d+1}(MT(d+2,2))\cong \Z/2}
}
\end{equation*}
commutes.
\end{proof}

Next we consider the groups $\pi_{d+2}(MT(d))$. Again the main problem is to determine the map $\beta^3$ in \eqref{flg}.


\begin{thm}
For $d\equiv 0 \bmod 4$ and $d \geq 4$, there is a short exact sequence
\begin{equation*}
0\to \Z/2\oplus \Z/2 \oplus \Z/2 \to \pi_{d+2}(MT(d)) \to \Omega_{d+2} \to 0.
\end{equation*}
\end{thm}

\begin{proof}
The sequence \eqref{flg} together with Theorem \ref{pi1MTd4} yields
\begin{equation*}
\pi_{d+3}(MT(d+3)) \xrightarrow{\beta^3} \pi_{d+3}(MT(d+3,3)) \to \pi_{d+2}(MT(d)) \to \Omega_{d+2} \to 0.
\end{equation*}
Since all manifolds of dimension $d+3$ allow three vector fields, $\beta^3$ vanishes on all of $\pi_{d+3}(MT(d+3))$. Thus the result follows from Theorem \ref{MTd3}.
\end{proof}

\begin{thm}
For $d \equiv 1 \bmod 4$ and $d\geq 5$, there is a short exact sequence
\begin{equation*}
0 \to \Z/2 \to \pi_{d+2}(MT(d)) \to \Omega_{d+2} \to 0.
\end{equation*} 
\end{thm}

\begin{proof}
In this case, $\pi_{d+3}(MT(d+3),3)\cong \Z \oplus \Z/4 \oplus \Z/2$. The composition 
\begin{equation*}
\pi_{d+3}(MT({d+3}))\xrightarrow{\beta^3} \pi_{d+3}(MT(d+3,3)) \xrightarrow{\Psi} KR(P_2) \cong \Z \oplus \Z/4
\end{equation*}
is surjective since \eqref{chi+sigma} is surjective. Thus the cokernel of $\beta^3$ can be at most $\Z/2$. 

In the Adams spectral sequence, $E_2^{0,d+3}(MT(d+3,3))\cong \Z/2 \oplus \Z/2$. The generators are represented by the fact that $w_{d+3}$ and $w_2w_{d+1}$ are not decomposable over $\A_2$ in $H^*(MT(d+3,3);\Z/2)$. We want to compare this to the Adams spectral sequence $E_2^{*,*}(MT(d+3))$ for $MT(d+3)$.
%

For $d+1 \equiv 2 \bmod 8$, there is a relation
\begin{equation*}
w_2w_{d+1} + w_{d+3}= \Sq^4(w_{d-1}) + \Sq^1(w_2w_d),
\end{equation*}
while for $d+1 \equiv 6 \bmod 8$, 
\begin{equation*}
w_2w_{d+1} = \Sq^4(w_{d-1}) + \Sq^1(w_2w_d)
\end{equation*}
in $H^*(MT(d+3))$. Thus the map $E_2^{0,d+3}(MT(d+3)) \to E_2^{0,d+3}(MT(d+3,3))$ is not surjective. Since there are no differentials in either of the spectral sequences in the relevant range, the map on $E_\infty^{0,d+3}$ is also not surjective. The element not hit cannot come from an element of $\pi_{d+3}(MT(d+3))$ of higher filtration, so the map $\pi_{d+3}(MT(d+3))\to \pi_{d+3}(MT(d+3,3))$ cannot be surjective. Hence the cokernel of $\beta^3$ must be $\Z/2$. 
\end{proof}

\begin{thm}
For all $d\equiv 2 \bmod 4$, there is a short exact sequence
\begin{equation*}
0\to \pi_{d+2}(MT(d)) \to \Omega_{d+2} \xrightarrow{\sigma} \Z/4 \to 0.
\end{equation*}
\end{thm}

\begin{proof}
In this case, $\pi_{d+3}(MT(d+3,3))\cong \Z/2 \oplus \Z/2$. The generators come from the fact that $w_2w_{d+1}$ and $w_{d+2}$ are indecomposable in $H^*(MT(d+3,3);\Z/2)$ and $\Sq^2(w_{d+2})=0$. These are also indecomposable in $H^*(MT(d+3);\Z/2)$ as one can see by computing 
\begin{align*}
\xi_5\xi_4^{\frac{d-2}{4}}(w_2w_{d+1})&=1\\
\xi_4^{\frac{d+2}{4}}(w_{d+2})&=1. 
\end{align*}
Here $\xi_i$ is the unique $\A_2$-homomorphism $\xi_i : H^i(MT(d+3);\Z/2) \to \Z/2$ for $i=4,5$ and the products of these are induced by the direct sum map $ BSO \times BSO \to BSO$.

This means that we get a surjection on $E_2$-terms of Adams spectral sequences that both collapse in the relevant range. Hence $\beta^3$ must be surjective. The existence of the short exact sequence then follows from previous results.

By the results of \cite{AD}, a manifold must 
satisfy $\sigma \equiv 0 \bmod 4$ if it has two independent vector fields. Since the image of $\pi_{d+2}(MT(d)) \to \Omega_{d+2}$ is the set of cobordism classes containing manifolds with two independent vector fields, this lies in the kernel of $\sigma : \Omega_{d+2} \to \Z/4$. It follows that the last map must be $\sigma$.
\end{proof}

\begin{thm}
For $d\equiv 3 \bmod 4$ and $d\geq 7$, there is a short exact sequence
\begin{equation*}
0 \to \Z/2\oplus \Z/2  \to \pi_{d+2}(MT(d)) \to \Omega_{d+2} \to 0.
\end{equation*}
\end{thm}

\begin{proof}
Now $\pi_{d+3}(MT(d+3,3))\cong \Z \oplus \Z/2 \oplus \Z/2$. Note that $\beta^3$ factors as
\begin{equation*}
\pi_{d+3}(MT(d+3)) \xrightarrow{\beta^5} \pi_{d+3}(MT(d+3,5)) \to \pi_{d+3}(MT(d+3,3)).
\end{equation*}
By Theorem \ref{MTd4}, $\pi_{d+3}(MT(d+3,5))\cong \Z/2\oplus \Z$. The Adams spectral sequences show that the $\Z/2$ summand is in the image of $\pi_{d+3}(MT(d,2))$, so it is mapped to 0 in $\pi_{d+3}(MT(d+3,3))$. From the spectral sequences, one also sees that the generator of the $\Z$ summand is mapped surjectively onto the generator of the $\Z$ summand in $\pi_*(MT(d+3,3))$.
This yields the short exact sequence claimed.
\end{proof}

Finally we consider the groups $\pi_{d+3}(MT(d))$. Using similar techniques, we obtain the following results:
\begin{thm}
For $d\equiv 0 \bmod 4$ and $d\geq 8$, there is a short exact sequence
\begin{equation*}
0\to \Z/24 \to \pi_{d+3}(MT(d)) \to \Omega_{d+3} \to 0.
\end{equation*}
\end{thm}

\begin{proof}
A very careful study of the the spectral sequences for the cofibration 
\begin{equation*}
MT(d,1)\to MT(d+4,5) \to MT(d+4,4)
\end{equation*}
shows that the image of $\beta^4$ is a subgroup of $\pi_{d+4}(MT(d+4,4))\cong \Z \oplus \Z/48 \oplus \Z/4$ isomorphic to $\Z \oplus \Z/8$ with cokernel $\Z/24$.
\end{proof}

\begin{thm}
For $d\equiv 1 \bmod 4$ and $d\geq 5$, there is a short exact sequence
\begin{equation*}
0\to \pi_{d+3}(MT(d)) \to \Omega_{d+3} \xrightarrow{\sigma} \Z/8 \to 0.
\end{equation*}
\end{thm}

\begin{proof}
Comparing the Adams spectral sequences shows that $\beta^4$ is surjective and the image of $\Omega_{d+3} \to \pi_{d+3}(MT(d+4,4))\cong \Z/8 \oplus \Z/2$ is $\Z/8$. The identification of the map $\Omega_{d+3} \xrightarrow{} \Z/8 $ again follows from \cite{AD} where it is shown that a manifold with three independent vector fields must satisfy $\sigma \equiv 0 \bmod 8$.
\end{proof}

For an oriented $d$-dimensional manifold $M$ and $w\in H^d(M;\Z/2)$ a product of Stiefel--Whitney classes for $TM$, the Stiefel--Whitney number $w[M]$ is given by evaluating $w$ on the fundamental class. 

\begin{thm}
For $d\equiv 2 \bmod 4$ and $d\geq 6$, there is a short exact sequence
\begin{equation*}
0 \to \Z/24 \to \pi_{d+3}(MT(d)) \to \Omega_{d+3} \to \Z/2 \to 0.
\end{equation*}
The last map is the Stiefel--Whitney number $w_2w_{d+1}[M]$.
\end{thm}

\begin{proof}
The exact sequence 
\begin{equation*}
\pi_{d+4}(MT(d,1)) \to \pi_{d+4}(MT(d+4,5)) \to \pi_{d+4}(MT(d+4,4))
\end{equation*}
is 
\begin{equation*}
\dotsm \to \Z\oplus \Z/2 \to \Z \oplus \Z/2 \to \Z \oplus \Z/24 \to \dotsm .
\end{equation*}
The Adams spectral sequences show that the first two $\Z/2$ and the last two $\Z$ summands  map isomorphically to each other, so the $\Z/24$ summand is not hit, thus it is not hit by $\beta^4$ either. The $\Z$ summand is hit by $\beta^4$ because it maps onto $2\Z \in \pi_{d+4}(MT(d+4,1)) \cong \Z$. This yields the short exact sequence.
\end{proof}

\begin{thm}
For $d\equiv 3 \bmod 4$ and $d \geq 7$, there is a short exact sequence
\begin{equation*}
0\to \Z/2 \oplus \Z/2 \to \pi_{d+3}(MT(d)) \to \Omega_{d+3} \to 0.
\end{equation*}
\end{thm}

\begin{proof}
Look at the exact sequence for the cofibration 
\begin{equation*}
MT(d+1,1) \to MT(d+4,4) \to MT(d+4,3).
\end{equation*}
Filling in the known groups, we see that $\pi_{d+4}(MT(d+4,4)) \to \pi_{d+4}(MT(d+4,3))$ is injective, so since $\beta^3 : \pi_{d+4}(MT(d+4)) \to \pi_{d+4}(MT(d+4,3))$ is zero, $\beta^4$ must also be zero. 
\end{proof}

\subsection{Identification of the global invariant for $r=4$}\label{ident}
In the above computations, we found descriptions of the invariants $\beta^4(M)$ which we summarize in the next theorem. The invariant $\chi_2$ is the mod 2 semi-characteristic, defined by the formula \eqref{chi2} except the cohomology groups should have $\Z/2$ coefficients.
\begin{thm}\label{beta4}
The following table displays the image $\Ima \beta^4$ of 
\begin{equation*}
\beta^4: \pi_d(MT(d))\to \pi_d(MT(d,4))
\end{equation*}
and the interpretation of $\beta^4$:
\begin{equation*}
\begin{tabular}{ccc}
\hline
 {$d \bmod 4$}&{$\Ima \beta^4$}&{$\beta^4$} \\
\hline
 {$ 0 $}&{$\Z\oplus \Z/8$}&{$\chi\oplus \frac{1}{2}(\chi+\sigma)$}\\

 {$1$}&{$\Z/2\oplus \Z/2 $}&{$\chi_2\oplus \chi_\R$}\\

 {$2 $}&{$\Z$}&{$\chi$}  \\

 {$ 3 $}&{$0$}&{$0$}\\
\hline
\end{tabular}
\end{equation*}
In particular, $\beta^4(M)$ is the top obstruction to $4$, $5$ and $6$ vector fields when $d$ is even.
\end{thm}
\begin{proof}
The groups $\Ima \beta^4$ were determined in the proofs above. For $d \equiv 0 \bmod 4$, we still need to identify $\beta^4(M)$. We saw that the image of $\beta^4$ was  $\Z \oplus \Z/8$. On the other hand,
 $\chi \oplus \frac{1}{2}(\chi+\sigma) : \pi_d(MT(d))\to \Z \oplus \Z/8$ vanishes for a manifold that allows four vector fields 
because of a theorem due to Mayer and Frank, see also~\cite{AD}, Corollary~6.6.
 Hence the kernel of $\beta^4$ is contained in the kernel of $\frac{1}{2}(\chi+\sigma)$ by Corollary~\ref{betared}. Thus there is a factorization
\begin{equation*}
\chi \oplus \frac{1}{2}(\chi+\sigma) : \pi_d(MT(d))\to \Ima (\beta^4) \to \Z \oplus \Z/8.
\end{equation*}
The composition is surjective, so the last map is forced to be an isomorphism.

For $d \equiv 1 \bmod 4$, it follows from the Adams spectral sequences that 
\begin{equation*}
\beta^4(M) = w_2w_{d-2}[M]\oplus \chi_\R(M).
\end{equation*}
By \cite{peterson}
\begin{equation*}
w_2w_{d-2}[M] = \chi_2(M) + \chi_\R(M).
\end{equation*}

For the last statement, it follows from the proof of Theorem \ref{injtr} that
 \begin{equation*}
\theta^r : \pi_{d-1}(V_{d,r}) \to \pi_d(MT(d,4))
 \end{equation*}
  is injective for $r=4,5,6$ and $d$ even. It maps the index $\ind(s) $ to $\beta^4(M)$ by construction.
\end{proof}

\begin{rem} 
For $d \equiv 0 \bmod 4$, the Atiyah--Dupont invariant is 
\begin{equation*}
\chi\oplus \frac{1}{2}(\chi \pm \sigma) \in \Z \oplus \Z/4.
\end{equation*}
Hence our invariant carries strictly more information in this case.
\end{rem}

\end{document}